\documentclass[12pt, a4]{amsart}

\usepackage{latexsym,amsmath,amssymb,amsthm,mathrsfs,color}

\usepackage[bbgreekl]{mathbbol}
\usepackage{bbm}
\usepackage{tikz-cd}
\usepackage[all]{xy}
\usepackage{stmaryrd}
\usepackage{hyperref}
\hypersetup{colorlinks,%
    citecolor=brown,%
    filecolor=black,%
   linkcolor=blue,%
   urlcolor=black}

\usepackage[toc,page]{appendix}
\usepackage[shortalphabetic]{amsrefs}
 \makeatletter
\def\append@label@year@{%
    \safe@set\@tempcnta\bib@year
    \edef\bib@citeyear{\the\@tempcnta}%
    \ifnum\bib@citeyear>9
      \append@to@stem{%
          \ifx\bib@year\@empty
          \else
            \@xp\year@short \bib@citeyear \@nil
          \fi
      }%
    \fi
}
\makeatother
\let\oldtocsection=\tocsection
\renewcommand{\tocsection}[2]{\hspace{0em}\oldtocsection{#1}{#2}}

\usepackage[a4paper]{geometry}

\def\upddots{\mathinner{\mkern 1mu\raise 1pt \hbox{.}\mkern 2mu
\mkern 2mu \raise 4pt\hbox{.}\mkern 1mu \raise 7pt\vbox {\kern 7
pt\hbox{.}}} }

\numberwithin{equation}{section}

\begin{document}
\setlength{\unitlength}{2.5cm}

%%%%%%%%%%% theorem styles
\newtheorem{thm}{Theorem}[section]
\newtheorem{lm}[thm]{Lemma}
\newtheorem{prop}[thm]{Proposition}
\newtheorem{cor}[thm]{Corollary}
\newtheorem{conj}[thm]{Conjecture}
\newtheorem{specu}[thm]{Speculation}

\theoremstyle{definition}
\newtheorem{dfn}[thm]{Definition}
\newtheorem{eg}[thm]{Example}
\newtheorem{rmk}[thm]{Remark}

\newcommand{\F}{\mathbf{F}}
\newcommand{\N}{\mathbf{N}}
\newcommand{\R}{\mathbf{R}}
\newcommand{\C}{\mathbf{C}}
\newcommand{\Z}{\mathbf{Z}}
\newcommand{\Q}{\mathbf{Q}}

\newcommand{\Mp}{{\rm Mp}}
\newcommand{\Sp}{{\rm Sp}}
\newcommand{\GSp}{{\rm GSp}}
\newcommand{\GL}{{\rm GL}}
\newcommand{\PGL}{{\rm PGL}}
\newcommand{\SL}{{\rm SL}}
\newcommand{\SO}{{\rm SO}}
\newcommand{\Spin}{{\rm Spin}}
\newcommand{\GSpin}{{\rm GSpin}}
\newcommand{\Ind}{{\rm Ind}}
\newcommand{\Res}{{\rm Res}}
\newcommand{\Hom}{{\rm Hom}}
\newcommand{\End}{{\rm End}}
\newcommand{\msc}[1]{\mathscr{#1}}
\newcommand{\mfr}[1]{\mathfrak{#1}}
\newcommand{\mca}[1]{\mathcal{#1}}
\newcommand{\mbf}[1]{{\bf #1}}

\newcommand{\mbm}[1]{\mathbbm{#1}}

\newcommand{\into}{\hookrightarrow}
\newcommand{\onto}{\twoheadrightarrow}

\newcommand{\s}{\mathbf{s}}
\newcommand{\cc}{\mathbf{c}}
\newcommand{\bfa}{\mathbf{a}}
\newcommand{\id}{{\rm id}}
\newcommand{\g}{\mathbf{g}_{\psi^{-1}}}
\newcommand{\w}{\mathbbm{w}}
\newcommand{\Ftn}{{\sf Ftn}}
\newcommand{\p}{\mathbf{p}}
\newcommand{\bq}{\mathbf{q}}
\newcommand{\WD}{\text{WD}}
\newcommand{\W}{\text{W}}
\newcommand{\Wh}{{\rm Wh}}
\newcommand{\ggma}{\omega}
\newcommand{\sct}{\text{\rm sc}}
\newcommand{\Of}{\mca{O}^\digamma}
\newcommand{\gk}{c_{\sf gk}}
\newcommand{\Irr}{ {\rm Irr} }
\newcommand{\Irrg}{ {\rm Irr}_{\rm gen} }
\newcommand{\diag}{{\rm diag}}
\newcommand{\uchi}{ \underline{\chi} }
\newcommand{\Tr}{ {\rm Tr} }
\newcommand{\der}\de
\newcommand{\Stab}{{\rm Stab}}
\newcommand{\Ker}{{\rm Ker}}
\newcommand{\bfp}{\mathbf{p}}
\newcommand{\bfq}{\mathbf{q}}
\newcommand{\KP}{{\rm KP}}
\newcommand{\Sav}{{\rm Sav}}
\newcommand{\de}{{\rm der}}
\newcommand{\tnu}{{\tilde{\nu}}}
\newcommand{\lest}{\leqslant}
\newcommand{\gest}{\geqslant}
\newcommand{\tu}{\widetilde}
\newcommand{\tchi}{\tilde{\chi}}
\newcommand{\tomega}{\tilde{\omega}}
\newcommand{\Rep}{{\rm Rep}}

\newcommand{\cu}[1]{\textsc{\underline{#1}}}
\newcommand{\set}[1]{\left\{#1\right\}}
\newcommand{\ul}[1]{\underline{#1}}
\newcommand{\wt}[1]{\overline{#1}}
\newcommand{\wtsf}[1]{\wt{\sf #1}}
\newcommand{\anga}[1]{{\left\langle #1 \right\rangle}}
\newcommand{\angb}[2]{{\left\langle #1, #2 \right\rangle}}
\newcommand{\wm}[1]{\wt{\mbf{#1}}}
\newcommand{\elt}[1]{\pmb{\big[} #1\pmb{\big]} }
\newcommand{\ceil}[1]{\left\lceil #1 \right\rceil}
\newcommand{\floor}[1]{\left\lfloor #1 \right\rfloor}
\newcommand{\val}[1]{\left| #1 \right|}

\newcommand{\exc}{ {\rm exc} }

\newcommand{\motimes}{\text{\raisebox{0.25ex}{\scalebox{0.8}{$\bigotimes$}}}}

%%%%%%%%% Imported from Dani's paper %%%%%%
\newcommand{\nequiv}{\not \equiv}
\newcommand{\half}{\frac{1}{2}}
\newcommand{\psii}{\widetilde{\psi}}
\newcommand{\ab} {|\!|}
\newcommand{\mb}{{\widetilde{B(\F)}}}

\title[Wavefront sets for theta representations]{On the wavefront sets associated with theta representations}

\author{Fan Gao and Wan-Yu Tsai}
\address{Fan Gao: School of Mathematical Sciences, Yuquan Campus, Zhejiang University, 38 Zheda Road, Hangzhou, China 310027}
\email{gaofan@zju.edu.cn}
\address{Wan-Yu Tsai: Department of Mathematics, National Tsing Hua University, Hsinchu 300, Taiwan}
\email{wytsai@math.nthu.edu.tw}

%\date{}
%\thanks{The second-named author is partially supported by the NSF grant DMS-1801273.
% The third-named author is partially supported by a Simons Foundation Collaboration Grant 426446.
%}
\maketitle

\begin{abstract} 
We study a conjectural formula for the maximal elements in the wavefront set associated with a theta representation of a covering group over $p$-adic fields. In particular, it is shown that the formula agrees with the existing work in the literature for various families of groups. We also recapitulate the results of an analogous formula in the archimedean case, which motivated the conjectural formula in the $p$-adic setting.
\end{abstract}
\tableofcontents

%%%
\section{Introduction}
Let $F$ be a local field of characteristic 0 with algebraic closure denoted by $\wt{F}$. Let $G$ be the $F$-rational points of a connected reductive group. Assume that $F^\times$ contains the full group $\mu_n$ of $n$-th roots of unity. In this paper we mainly consider a central cover 
$$\begin{tikzcd}
\mu_n \ar[r, hook] & \wt{G} \ar[r, two heads] & G
\end{tikzcd}$$
 of $G$ by $\mu_n$ arising from the Brylinski--Deligne framework \cite{BD}.

Every irreducible admissible representation $(\pi, V_\pi)$ of $\wt{G}$ defines a character distribution $\chi_\pi$ in a neighborhood of $0$ in $\mfr{g}={\rm Lie}(G)$. Assume that $F$ is a $p$-adic field and $G$ is split. It follows from the work of Howe \cite{How1} and Harish-Chandra \cite{HC1} for linear groups and its extension to covering groups by W.-W. Li \cite{Li3} that there exists a compact open subset $S_\pi$ of $0$ such that for every smooth function $f$ with compact support in $S_\pi$, one has
\begin{equation} \label{E:char}
\chi_\pi(f) = \sum_{\mca{O} \in \mca{N}}c_\mca{O} \cdot \int \hat{f} \ \mu_\mca{O},
\end{equation}
where $\mca{N}$ denotes the set of nilpotent orbits in $\mfr{g}$ under the conjugation action of $G$. Here $\mu_\mca{O}$ is a certain Haar measure on $\mca{O}$ properly normalized, and $\hat{f}$ is the Fourier transform of $f$ with respect to the Cartan--Killing form on $\mfr{g}$ and a non-trivial character 
$$\psi_\natural: F\to \C^\times;$$
one has $c_\mca{O}:=c_{\mca{O}, \psi_\natural} \in \C$. See the references mentioned above and \cite{MW1, Var1}. An analogue of \eqref{E:char} for $F=\R$ was given by Barbasch--Vogan \cite{BV3}, see the discussion in \S \ref{s:WFAV}.

Denote
$$\mca{N}_{\rm tr}(\pi) = \set{\mca{O} \in \mca{N}: \ c_\mca{O} \ne 0}$$
and let
$$\mca{N}_{\rm tr}^{\rm max}(\pi) \subset \mca{N}_{\rm tr}(\pi)$$
be the subset consisting of all maximal elements in $\mca{N}_{\rm tr}(\pi)$ with respect to the partial order $\mca{O}_1 \lest \mca{O}_2$, which is defined by $\mca{O}_1 \subset \overline{\mca{O}}_2$. Here $\wt{\mca{O}}$ denotes the topological closure of $\mca{O}$. The wavefront set of $\pi$ is given by
$${\rm WF}(\pi) = \bigcup_{\mca{O} \in \mca{N}_{\rm tr}(\pi)} \wt{\mca{O}},$$
which contains elements in $\mca{N}_{\rm tr}^{\rm max}$ as the maximal nilpotent classes. It is known that the set $\mca{N}_{\rm tr}^{\rm max}(\pi)$ is equal to the set of maximal nilpotent orbits with respect to which the generalized Whittaker models for $\pi$ are nontrivial, see \cite{MW1, Var0, Pate}. 

While the two sets $\mca{N}_{\rm tr}(\pi)$ and $\set{c_\mca{O}: \mca{O} \in \mca{N}_{\rm tr}(\pi)}$ contain deep and important arithmetic and representation-theoretic information of $\pi$, they are not easily accessible and there is no simple formula to compute them, see for instance \cite{HII, HIIc, HC1} and references therein. The reader is also referred to the work of Ginzburg \cite{Gin0, Gin2} for global analogue and open questions for automorphic representations. 

A particularly important case is that $\mca{N}_{\rm tr}^{\rm max}(\pi)$, when taken closure in $\mfr{g}\otimes \wt{F}$, is the minimal orbit $\mca{O}_{\rm min}$, whence $\pi$ is called a minimal representation \cite{Tor, GS05}.
In fact, for the existence of such a minimal representation, it is necessary to relax the condition so that $\pi$ may be a genuine representation of a finite degree central cover of $G$. For example, for $\Sp_{2r}$ the Weil representation is a minimal representation defined on the double cover $\wt{\Sp}_{2r}^{(2)}$. For simply-laced groups $G$, minimal representations are studied and constructed (whenever they exist on $G$) by Kazhdan and Savin \cite{KS, Sav94}. For exceptional groups one can refer to the work as in \cite{Sav4, Rum, LoSa1}. 

Such minimal representations play a pivotal role in various instances of liftings of representations. Indeed, the Weil representation enables the classical theory of theta liftings. A cubic correspondence between $\wt{\SL}_2^{(3)}$ and $\SL_2$, both local and global, was established by using the minimal representations on $\wt{G}_2^{(3)}$, see \cite{GRS1}.  For liftings using the minimal representation on $\wt{F}_4^{(2)}$, see \cite{Gin5}.

Even if $\pi$ of $\wt{G}$ is not a minimal representation, it is possible to analyze the decomposition of $\pi$ to a pair of mutual centralizer subgroups and thus have an analogous lifting of representations, as long as the representation is ``small" enough, in the sense that $\mca{N}_{\rm tr}^{\rm max}(\pi)$ contains small nilpotent orbits. Such small representations are most often the residues of certain Borel Eisenstein series or closely related representations. When the degree $n$ of the covering is small, it is expected that such a residual representation is small. For example, such small representations and its entailed theory of liftings for $\wt{\SO}_{2r+1}^{(4)}$ were discussed in \cite{BFrG2, BFrG, LoSa2}; similar analysis for the double cover $\wt{\GSp}_{2r+1}^{(2)}$ was also carried out in \cite{Kap004}.

For degree-four cover of $\Sp_{2r}$, a theory of theta lifting was investigated by Leslie \cite{Les}. For $\wt{\Sp}_{2r}^{(n)}$ with odd $n$, the theta representations were studied in depth by Friedberg and Ginzburg towards a theory of generalized theta liftings and descent on such high degree covers of $\Sp_{2r}$, see the series of papers \cite{FG2, FG4, FG5, FG6}. In particular, in \cite{FG5, FG6} the authors gave a conjectural description of $\mca{N}_{\rm tr}^{\rm max}(\Theta(\wt{\Sp}_{2r}^{(n)} ))$ for the theta representation of $\wt{\Sp}_{2r}^{(n)}$, which asserts that 
$$\mca{N}_{\rm tr}^{\rm max}(\Theta(\wt{\Sp}_{2r} ^{(n)})) \otimes \wt{F}:=\set{\mca{O}\otimes \wt{F}: \ \mca{O} \in \mca{N}_{\rm tr}^{\rm max}(\Theta(\wt{\Sp}_{2r} ^{(n)}))}$$
equals the symplectic collapse of the partition $(n^q t)$ of $2r$, where $2r=qn + t$ with $0\lest t <n$. In fact, this conjectural formula for $\mca{N}_{\rm tr}^{\rm max}(\Theta(\wt{\Sp}_{2r} ^{(n)})) \otimes \wt{F}$  could be viewed as a natural analogue for the case of Kazhdan--Patterson covers $\wt{\GL}_r^{(n)}$. Indeed, it was shown by Savin \cite{Sav2} and Y.-Q. Cai \cite{Cai1} that
$$\mca{N}_{\rm tr}^{\rm max}(\Theta(\wt{\GL}_r^{(n)})) = \set{(n^q t)},$$
where $r=qn + t$ with $0\lest t < n$.

\subsection{Main result}
Since it is important to determine the set $\mca{N}_{\rm tr}^{\rm max}(\Theta(\wt{G}))$ for a covering group, the goal of our paper is to point out a unified (conjectural) recipe of computing this set and also $c_\mca{O}$ with $\mca{O} \in \mca{N}_{\rm tr}^{\rm max}(\Theta(\wt{G}))$ for a persistent covering group. The formula for $\mca{N}_{\rm tr}^{\rm max}(\Theta(\wt{G})) \otimes \wt{F}$ relies on a certain Macdonald representation and the Springer correspondence, see Conjecture \ref{C:main} for details. The main results constitute the following:
\begin{enumerate}
\item[(i)] In \S \ref{S:generic}, we  verify parts of Conjecture \ref{C:main} for generic $\Theta(\wt{G})$, i.e., when it possesses a nontrivial Whittaker model, see Theorem \ref{T:generic}. In fact, we also state a natural generalization of Conjecture \ref{C:main} for all constituents of an unramified regular principal series, see Conjecture \ref{C:main-g}.
\item[(ii)] In \S \ref{S:p}, we show that Conjecture \ref{C:main} is compatible with the work of  Friedberg--Ginzburg and Y.-Q. Cai for  $\wt{\Sp}_{2r}$ and $\wt{\GL}_r$ respectively, see Theorems \ref{T:Sp} and \ref{T:GL}. In addition, for $\wt{\Sp}_{2r}^{(4)}, \wt{\SO}_{2r+1}^{(4)}$ and $\wt{\rm GSpin}_{2r+1}^{(2)}$, we also show that Conjecture \ref{C:main} agrees with the respective work \cite{Les}, \cite{BFrG2} and \cite{Kap004}, as mentioned above; see the end of \S \ref{S:p}. 
\item[(iii)] In \S \ref{S:arch}, we consider the archimedean analogue of Conjecutre \ref{C:main}, and recapitulate the relevant work in the literature, see Theorem \ref{p:diagram} for the main result. As an application, we discuss about the complex and real case separately. A special family in the real case concerns the Langlands quotients of certain pseudospherical principal series studied in \cite{ABPTV}, which are just the theta representations we focus in this paper. The results on such family of representations, as discussed in  \S \ref{SSS:R-spl}, gave us the initial motivation to the formulation of Conjecture \ref{C:main}.
\end{enumerate}
  
In the last section \S \ref{S:rmk}, we remark on several (lacking) aspects of the formulation in Conjecture \ref{C:main}, including some possibility on the uniformization of a statement for all local fields, and on the relation with the generalized Whittaker space as studied in \cite{MW1, GGS17, GGS} and so on.

In this paper, we follow the standard notations and terminologies from the work of Brylinski--Deligne \cite{BD} and also those in \cite{We6, GG}. We also use the following notations:
\begin{enumerate}
\item[--]  $\floor{x} \in \Z$ to denote the integral part of $x\in \R$,
\item[--] $[a, b]$ to denote $[a, b] \cap \Z$ for every $a, b \in \Z$.
\end{enumerate}

%%
%\subsection{Acknowledgement}

%%%
\section{Wavefront set for $\Theta(\pi^\dag, \nu)$} \label{S:main}
Let $F$ be a local field of characteristic 0 endowed with a valuation $|\cdot|_F$. If $F$ is $p$-adic, then we denote by $O_F \subset F$ the ring of integers and by $\varpi$ a fixed uniformizer. Let $G$ be the $F$-rational points of a split reductive group over $F$. Denote by 
$$(X,\ \Phi, \ \Delta; \ Y, \Phi^\vee, \Delta^\vee)$$
 the root datum of $G$, where $X$ is the character lattice and $Y$ the cocharacter lattice of a split torus $T\subset G$. Here $\Delta$ is a choice of simple roots, and $Y^{sc} \subset Y$ is the coroot sublattice and $X^{sc} \subset X$ the root lattice.

Let $$Q: Y \longrightarrow \Z$$ be a Weyl-invariant quadratic form, and let $B_Q$ be the associated bilinear form. One has an associated covering group  $\wt{G}:=\wt{G}_Q$ of $G$ by $\mu_n$. The dual group $\wt{G}^\vee$ has a root datum
$$(Y_{Q,n}, \ \Delta_{Q,n}^\vee;\ X_{Q,n},\ \Delta_{Q,n}).$$
Here $Y_{Q,n} \subset Y$ is the character lattice of $\wt{G}^\vee$ given by 
$$Y_{Q,n}=\set{y\in Y: \ B_Q(y, z)\in n\Z \text{ for all } z \in Y},$$
and $X_{Q,n}:=\Hom(Y_{Q,n}, \Z)$. The set $\Delta_{Q,n}^\vee$ consists of the re-scaled simple coroots 
$$\alpha_{Q,n}^\vee:=n_\alpha \alpha^\vee:=\frac{n}{\gcd(n,Q(\alpha^\vee))} \alpha^\vee.$$
Let $Y_{Q,n}^{sc} \subset Y_{Q,n}$ be the sublattice spanned by $\Delta_{Q,n}^\vee$.

\subsection{Theta representations}  \label{SS:Theta}
We have the following relations among various real vector spaces:
\begin{equation} \label{Eq:nu}
\begin{tikzcd}
\Hom(Y, \R) \ar[r, equal] & X\otimes \R \ar[r, "f", "\simeq"'] & X_{Q,n} \otimes \R \ar[r, equal] & \Hom(Y_{Q,n}, \R) \ar[r, two heads, "\phi"] & \Hom(Y_{Q,n}^{sc}, \R).
\end{tikzcd}
\end{equation}
Here $f$ is an isomorphism since $X$ is a sublattice of $X_{Q,n}$ of the same rank, and the surjectivity of $\phi$ follows from the elementary divisor theorem for the pair of lattices $Y_{Q,n}^{sc} \subset Y_{Q,n}$. If $Y_{Q,n}$ and $Y_{Q,n}^{sc}$ have the same rank, for example when $G$ is semsimple, then $\phi$ is also an isomorphism. We identify the first four real vector spaces in \eqref{Eq:nu}.

For every $\nu \in X\otimes \R$, there is a map
$$\delta_\nu: T \longrightarrow \C^\times$$
given by 
$$\delta_\nu(y\otimes a) = |a|_F^{\nu(y)}$$
on the generators $y\otimes a\in T$, where $\nu(y)$ is the natural pairing between $Y$ and $X\otimes \R$.

Let $\wt{T} \subset \wt{G}$ be the covering torus of $\wt{G}$. We assume that there exists a certain distinguished (finite-dimensional) genuine representation $\pi^\dag$ of $\wt{T}$ determined by a distinguished genuine central character $\chi^\dag$ of $Z(\wt{T})$; see \cite[\S6]{GG} for a discussion on the necessary conditions for its existence. It is also shown in \cite[Theorem 6.6]{GG} that the set of all distinguished central characters of $Z(\wt{T})$ (and thus also distinguished representations of $\wt{T}$), whenever exists, is a torsor over
$$\Hom(F^\times/F^{\times n}, Z^\heartsuit(\wt{G}^\vee)),$$
where 
$$Z^\heartsuit(\wt{G}^\vee):=\Hom(Y_{Q,n}/(nY + Y_{Q,n}^{sc}), \C^\times) \subset Z(\wt{G}^\vee).$$
Note that if $G$ is simply-connected and semisimple, then $nY \subset Y_{Q,n}$ and thus $Z^\heartsuit(\wt{G}^\vee) = Z(\wt{G}^\vee)$.
The character $\chi^\dag$ is always Weyl-invariant and satisfies
$$\chi^\dag(\wt{h}_\alpha(a^{n_\alpha})) =1$$
for every $a\in F^\times$ and $\alpha\in \Delta$. 
Relying on a nontrivial additive character $\psi: F\to \C^\times$ and thus the associated Weil index, a detailed construction is given in \cite[\S 7]{GG} of such a distinguished central character denoted by $\chi_\psi$ and the distinguished representation $\pi_\psi$.  
% In the non-archimedean tame case (i.e., when $p\nmid n$), we assume that the conductor of $\psi$ satisfies $\mfr{f}(\psi)=O_F$ and $\chi_\psi$ is unramified.

For every $\nu\in X\otimes \R$, denote by
$$I(\pi^\dag, \nu):=\Ind_{\wt{B}}^{\wt{G}} (\pi^\dag \otimes \delta_\nu)$$
the normalized induced principal series of $G$. If $F$ is non-archimedean, then the space $I(\pi^\dag, \nu)$ consists of locally constant functions on $G$ valued in the finite-dimensional space $\pi^\dag \otimes \delta_\nu$. If $F$ is archimedean, we further take the $K$-finite smooth vectors which afford the structure as a $(\mfr{g}, K)$-module, where $K \subset G$ is a fixed maximal compact subgroup. By abuse of notation, we still use $I(\pi^\dag, \nu)$ to denote this $(\mfr{g}, K)$-module. For $F=\R$, the representation $I(\pi^\dag, \nu)$ is just the one studied in \cite{ABPTV}, see also the discussion in \S \ref{SSS:R-spl}.

\begin{dfn} \label{D:exc}
A vector $\nu \in X\otimes \R$ is called an exceptional character if $\nu(\alpha_{Q,n}^\vee) =1$ for every $\alpha\in \Delta$.
\end{dfn}
It follows from the Langlands classification theorem for covers (see \cite{BJ1}) that if $\nu\in X\otimes \R$ is exceptional, then we have
$$I(\pi^\dag, \nu) \onto \Theta(\pi^\dag, \nu),$$
where $\Theta(\pi^\dag, \nu)$ is the unique Langlands quotient of $I(\pi^\dag, \nu)$.

\begin{eg} \label{Eg:nu}
By definition, the map $\nu \mapsto \phi(\nu)$ is constant on the exceptional characters in $X\otimes \R$, and in fact
$$\phi(\nu) = \sum_{\alpha\in \Delta} (\omega_\alpha/n_\alpha) \in X^{sc}\otimes \R \subset X\otimes \R,$$
where $\omega_\alpha$ and $\omega_\alpha/n_\alpha$ are the fundamental weights associated with $\alpha^\vee$ and $\alpha_{Q,n}^\vee$ respectively.  If the root system of $G$ is of simply-laced type, then 
$$\phi(\nu) = \rho/n_\alpha \in X\otimes \R,$$
where $\rho = \sum_{\alpha \in \Delta} \omega_\alpha$.
\end{eg}

\subsection{Saturated and persistent covers}
A covering group $\wt{G}$ is called saturated (see \cite[Definition 2.1]{Ga6}) if 
$$Y^{sc} \cap Y_{Q,n} = Y_{Q,n}^{sc},$$
where the one-sided inclusion $\supset$ always holds. In general, for every $\alpha\in \Phi$ one has
$$\Z[\alpha^\vee] \cap Y_{Q,n} = \Z[i_\alpha \cdot \alpha_{Q,n}^\vee] \text{ with } i_\alpha \in \set{1, 1/2},$$
and $i_\alpha =1/2$ only if $n_\alpha$ is even. Set
$$\tilde{n}_\alpha = i_\alpha \cdot n_\alpha, \ \tilde{\alpha}_{Q,n}^\vee=\tilde{n}_\alpha \cdot \alpha^\vee, \text{ and } \tilde{\alpha}_{Q,n} = \alpha/\tilde{n}_\alpha$$
for every $\alpha \in \Phi$, and
$$\tilde{\Phi}_{Q,n}^\vee:=\set{\tilde{\alpha}_{Q,n}^\vee: \ \alpha\in \Phi}.$$
Let 
$$\tilde{Y}_{Q,n}^{sc} \subset Y_{Q,n}$$
 be the sublattice spanned by $\tilde{\Phi}_{Q,n}$, and we call it the saturation of $Y_{Q,n}^{sc}$. One has
$$Y_{Q,n}^{sc} \subset \tilde{Y}_{Q,n}^{sc} \subset Y_{Q,n},$$
and if $\wt{G}$ is saturated, then $Y_{Q,n}^{sc} = \tilde{Y}_{Q,n}^{sc}$. However, the converse may not hold, for an example, see $\wt{\SL}_3^{(3)}$ whose dual group is $\SL_3$. In fact, we have an essentially complete understanding of the case $Y_{Q,n}^{sc} \subsetneq \tilde{Y}_{Q,n}^{sc}$ as follows.

\begin{lm} \label{L:clas}
Let $\alpha \in \Delta$. If $\Z\alpha^\vee \cap Y_{Q,n} = \Z[\alpha_{Q,n}^\vee/2]$, then necessarily $\alpha$ is a long simple root, $2|\angb{\alpha}{y}$ for all $y\in Y$, and thus the root system of $G$ is of type $C_r$. In particular, if $G$ is semisimple, then the only cover $\wt{G}$ such that $Y_{Q,n}^{sc} \subsetneq \tilde{Y}_{Q,n}^{sc}$ satisfies exactly the following:
\begin{enumerate}
\item[--] $G=\Sp_{2r}$ and $n_{\alpha_r}$ is even for the unique long simple root $\alpha_r$,
\item[--] $ \tilde{Y}_{Q,n}^{sc}= Y_{Q,n} =(n_{\alpha_r}/2)\cdot Y$, while $Y_{Q,n}^{sc}$ is spanned by $\set{(n_{\alpha_r}/2)\cdot \alpha_i^\vee: \alpha_i \in \Delta \text{ is short}} \cup \set{n_{\alpha_r}\alpha_r^\vee}$.
\end{enumerate}
In any case, for an arbitrary $\wt{G}$ the set $\tilde{\Phi}_{Q,n}^\vee$ forms a root system.
\end{lm}
\begin{proof}
We consider the unique simple root of $\SL_2$ as being long. Suppose $Y_{Q,n}^{sc} \subsetneq \tilde{Y}_{Q,n}^{sc}$, then $n_\alpha \alpha^\vee/2 \in Y^{sc} \cap Y_{Q,n}$
for some $\alpha \in \Delta$. This implies that 
$$\frac{n_\alpha}{2} B_Q(\alpha^\vee, y) = \frac{n_\alpha Q(\alpha^\vee)}{2} \angb{\alpha}{y}  \in n_\alpha \Z$$
 for all $y\in Y$, i.e., $2|\angb{\alpha}{y}$. If the semisimple rank of $G$ is one, then the assertions are clear.  Assume the semisimple rank of $G$ is at least two. Then $2|\angb{\alpha}{\beta}$ for all $\beta \in \Delta$. This is possible only if the root system of $G$ is of type $C_r$ and that $\alpha$ is the unique long root.
 
 If $G$ is semisimple, then necessarily $G=\Sp_{2r}$ or ${\rm PGSp}_{2r}$. However, in the latter case the fundamental coweight $\omega_r^\vee$ of $\alpha_r$ lies in $Y$ and we have $\angb{\alpha}{\omega_r^\vee} =1$. Thus, we must have $G=\Sp_{2r}$. The rest of the assertions follows easily from this.
\end{proof}

\begin{dfn}
An element $\tnu \in X\otimes \R$ is called a saturation of an exceptional character $\nu \in X\otimes \R$ if $\tilde{\nu}(\tilde{\alpha}_{Q,n}^\vee) =1$ for every $\alpha\in \Delta$.
\end{dfn}
If $\wt{G}$ is saturated, then $\phi(\tnu) = \phi(\nu) \in \Hom(Y_{Q,n}^{sc}, \R)$ for every saturation $\tnu$ of an exceptional $\nu$.

We also recall the notion of a persistent cover as follows (see \cite[Definition 2.3]{Ga6}). Consider
$$\msc{X}_{Q,n}^{sc}:=Y/Y_{Q,n}^{sc}, \ \msc{X}_{Q,n}=Y/Y_{Q,n},$$
 which are both endowed with the twisted Weyl action 
 $$w[y]:=w(y-\rho^\vee) + \rho^\vee.$$
 Here $\rho^\vee$ is the half sum of all positive coroots in $\Phi^\vee$. For every $y\in Y$, let $y^\dag$ and $y^\ddag$ denote its image in $\msc{X}_{Q,n}^{sc}$ and $\msc{X}_{Q,n}$ respectively.  An $n$-fold cover $\wt{G}$ is called persistent if 
$${\rm Stab}_W(y^\dag; \msc{X}_{Q,n}^{sc}) = {\rm Stab}_W(y^\ddag; \msc{X}_{Q,n})$$
for every $y\in Y$. While persistency is a slightly technical condition, we note the following:
\begin{enumerate}
\item[$\bullet$] a saturated cover is always persistent, 
\item[$\bullet$] if $G$ is semisimple and simply-connected, then $G$ is saturated if and only if its dual group $\wt{G}^\vee$ is of adjoint type, i.e., $Y_{Q,n} = Y_{Q,n}^{sc}$.
\end{enumerate}
As another example, every cover of $\GL_r$ is saturated and thus persistent. On the other hand, the cover $\wt{\SL}_2^{(n)}$ associated with $Q(\alpha) = -1$ is saturated if $n$ is odd, and is persistent but not saturated if $4|n$; if $n\in 4\Z +2$, then $\wt{\SL}_2^{(n)}$ is not persistent. For more examples of saturated covers, we refer to \cite[\S 2.7]{We6}.

\subsection{The set $\mca{N}_{\rm tr}^{\rm max}(\Theta(\pi^\dag, \nu))$} 
To state the conjectural formula, we first briefly recall the Macdonald representation of a Weyl group and the Springer correspondence.

For every $\nu \in X\otimes \R$, denote by
$$W_\nu=\set{w\in W: \ w(\nu) - \nu \in X^{sc}} \subset W$$
the integral Weyl subgroup associated with $\nu$. It is a reflection subgroup associated with the root subsystem
$$\Phi_\nu = \set{\alpha\in \Phi: \ \angb{\nu}{\alpha^\vee} \in \Z}.$$
In fact, the set 
$$\Phi_\nu^+:=\Phi_\nu \cap \Phi^+$$ is a set of positive roots of $\Phi_\nu$, and we let $\Delta_\nu \subset \Phi_\nu^+$ be the associated simple roots for the system $\Phi_\nu$. Note that in general $\Delta_\nu \ne \Phi_\nu \cap \Delta$.

One has a  composite of canonical surjections
$$\begin{tikzcd}
\Hom(Y_{Q,n}, \R) \ar[r, two heads, "\tilde{\phi}"] \ar[rr, bend right =15, "\phi"'] & \Hom(\tilde{Y}_{Q,n}^{sc}, \R) \ar[r, two heads] & \Hom(Y_{Q,n}^{sc}, \R),
\end{tikzcd}$$
which is $W$-equivariant with respect to the usual reflection action.
We have 
$$\Ker(\tilde{\phi}) \subset \Ker(\phi),$$
both of which are fixed by the Weyl group pointwise. Thus, for every $\nu \in X\otimes \R$ the root subsystem $\Phi_\nu$ and $W_\nu$ depend only on $\phi(\nu)$ or $\tilde{\phi}(\nu)$; that is, $W_\nu$ is actually equal to the integral Weyl subgroup associated
with $\tilde{\phi}(\nu)$ and $\phi(\nu)$, with respect to the $W$-action on $\Hom(\tilde{Y}_{Q,n}^{sc}, \R)$ and $\Hom(Y_{Q,n}^{sc}, \R)$, respectively.

Let $\varepsilon_\nu = \varepsilon_{W_\nu}$ be the sign character of $W_\nu$. The construction of the Macdonald representation $j_{W_\nu}^W(\varepsilon_\nu) \in \Irr(W)$ arising from the truncated $j$-induction (see \cite{Mac1, LuSp1} or \cite[\S 11.2]{Car}) is given as follows. First, we have
$$\wp_\nu=\prod_{\alpha \in \Phi_\nu^+} \alpha,$$
which is a homogeneous rational-valued polynomial on $Y\otimes \R$. Let 
$$P(\Phi_\nu) =\set{w(\wp_\nu): \ w\in W}$$
be the subspace of the symmetric algebra ${\rm Sym}(X\otimes \R)$ spanned by the $w(\wp_\nu)$'s. It is shown in \cite{Mac1} that $P(\Phi_\nu)$ affords an irreducible representation of $W$ which we denote by 
$$j_{W_\nu}^W(\varepsilon_\nu).$$
In fact, $j_{W_\nu}^W(\varepsilon_\nu)$ is the unique subrepresentation of $\Ind_{W_\nu}^W(\varepsilon_\nu)$ governed by the leading term of a certain fake degree polynomial associated with the natural reflection representation of $W$ (see \cite[\S 11.1]{Car}). The two special cases are $W_\nu = \set{1}$ and $W_\nu = W$, for which the representation $j_{W_\nu}^W(\varepsilon_\nu)$ equals $\mbm{1}$ and $\varepsilon_W$ respectively.

To recall the Springer correspondence \cite{Spr}, let $\mfr{g}\otimes \wt{F}$ be the Lie algebra of $G$ over the algebraically closed field $\wt{F}$. Let $\mfr{B}$ be the flag variety of all Borel subalgebras of $\mfr{g}\otimes \wt{F}$. For a nilpotent element $x\in \mfr{g}\otimes \wt{F}$ one has the subvariety $\mfr{B}_x$ of Borel subalgebras containing $x$. The group $G_{ad}^x$, which is the stabilizer of $x$ in $G_{ad}$, acts on $\mfr{B}_x$. One has a well-defined action of $G_{ad}^x$ on $H^*(\mfr{B}_x, \wt{F})$ which factors through the component group 
$$A_x:=G_{ad}^x/(G_{ad}^x)^o.$$
There is a natural action of $W$ on $H^*(\mfr{B}_x, \wt{F})$ which commutes with that of $A_x$. This gives a decomposition of the top degree cohomology space
$$H^{\rm top}(\mfr{B}_x, \wt{F}) = \bigoplus_{\eta \in \Irr(A_x)} \eta \boxtimes \sigma_\eta,$$
where $\sigma_\eta \in \set{0} \cup \Irr(W)$. There are many properties of the correspondence thus established, one of which concerns us is that every $\sigma \in \Irr(W)$ is isomorphic to $\sigma_\eta$ for a unique nilpotent orbit $\mca{O}_x$ and a unique $\eta \in \Irr(A_x)$. In fact, $A_x$ depends only on the conjugacy class $\mca{O}_x$ of $x$ and thus we may write $A_{\mca{O}}$ for $A_x$ for any $x\in \mca{O}$.
Defining
$$\mca{N}^{\rm en}=\set{(\mca{O}, \eta): \ \mca{O} \in \mca{N} \text{ and } \eta \in \Irr(A_\mca{O})},$$
we thus obtain an injective map
$$\begin{tikzcd}
{\rm Spr}: \Irr(W) \ar[r, hook] & \mca{N}^{\rm en}
\end{tikzcd}$$
denoted by
$${\rm Spr}(\sigma)=(\mca{O}_{\rm Spr}(\sigma), \eta(\sigma));$$
we call
$$\mca{O}_{\rm Spr}(\sigma) \subset \mfr{g}\otimes \wt{F}$$
the nilpotent orbit associated with $\sigma$. In particular, we have  $\mca{O}_{\rm Spr}(\mbm{1}) = \mca{O}_{\rm reg}$, the regular orbit; on the other hand, $\mca{O}_{\rm Spr}(\varepsilon_W) = \mca{O}_0$, the trivial orbit. Note that for every $\mca{O} \in \mca{N}$, the pair $(\mca{O}, \mbm{1})$ lies in the image of ${\rm Spr}$, i.e., $(\mca{O}, \mbm{1}) = {\rm Spr}(\sigma_\mca{O})$ for a unique $\sigma_\mca{O} \in \Irr(W)$. This gives us a well-defined injective map
$$\begin{tikzcd}
{\rm Spr}_\mbm{1}^{-1}: \mca{N} \ar[r, hook] & \Irr(W) 
\end{tikzcd}$$
given by 
$${\rm Spr}_\mbm{1}^{-1}(\mca{O}):={\rm Spr}^{-1}((\mca{O}, \mbm{1})).$$
It is clear that
$$\mca{O}_{\rm Spr} \circ {\rm Spr}_\mbm{1}^{-1} = \text{id}_{\mca{N}};$$
however, $ {\rm Spr}_\mbm{1}^{-1}\circ \mca{O}_{\rm Spr}$ may not be the identity map on $\Irr(W)$.

One has the permutation representation
$$\sigma^\msc{X}: W \longrightarrow {\rm Perm}(\msc{X}_{Q,n})$$
given by the twisted Weyl action $w[y] = w(y-\rho^\vee) + \rho^\vee$.

\begin{conj} \label{C:main}
Let $F$ be $p$-adic with $p\nmid n$. Let $\wt{G}$ be a persistent $n$-fold covering group. Let $\nu \in X\otimes \R$ be exceptional and let $\tilde{\nu} \in X\otimes \R$ be a saturation of $\nu$. Then for the Harish-Chandra local character expansion of $\Theta(\pi^\dag, \nu)$ as in \eqref{E:char}, one has
\begin{equation} \label{E:main1}
\mca{N}_{\rm tr}^{\rm max}(\Theta(\pi^\dag, \nu))\otimes \wt{F} = \set{ \mca{O}_{\rm Spr}(j_{W_\tnu}^W(\varepsilon_\tnu)) }
\end{equation}
and  
\begin{equation} \label{E:main2}
c_{\mca{O}} = \angb{j_{W_\tnu}^W(\varepsilon_\tnu)}{ \varepsilon_W \otimes \sigma^\msc{X} }_W
\end{equation}
for every orbit $\mca{O} \in \mca{N}_{\rm tr}^{\rm max}(\Theta(\pi^\dag, \nu))$.
\end{conj}
The various measures involved in the local character expansion \eqref{E:char} are chosen as in \cite{HC1, MW1, Var0, Li3, Pate}. We also remark the following:
\begin{enumerate}
\item[$\bullet$] It is part of Conjecture \ref{C:main} that elements in $\mca{N}_{\rm tr}^{\rm max}(\Theta(\pi^\dag, \nu))$ all lie in one single $\wt{F}$-nilpotent orbit, and it is a delicate issue to determine the $F$-nilpotent classes in $\mca{N}_{\rm tr}^{\rm max}$. For arbitrary irreducible representation of the linear group $\wt{G}=G$, this was first conjectured by M\oe glin and Waldspurger, and it is expect to hold for representations of $\wt{G}$ besides $\Theta(\pi^\dag, \nu)$.
\item[$\bullet$] Implicit in Conjecture \ref{C:main} is that for persistent covers,  the equalities \eqref{E:main1} and \eqref{E:main2} are independent of the choice of distinguished representation $\pi^\dag$ of $\wt{T}$, and also independent of the nontrivial $\psi_\natural$, which is used in giving the character expansion \eqref{E:char}. In fact, if $\wt{G}^\vee$ has trivial center, which in particular implies that $\wt{G}$ is saturated, then there is a unique distinguished representation $\pi^\dag$ of $\wt{T}$, see the discussion in \S \ref{SS:Theta}. 
\item[$\bullet$] If $\wt{G}$ is not persistent, then \eqref{E:main1} is still expected to hold, but the determination of the set $\mca{N}_{\rm tr}^{\rm max}(\Theta(\pi^\dag, \nu))$ is more complicated, involving subtle relations between $\pi^\dag$ and $\psi_\natural$. This already occurs for the double cover of $\SL_2$, see \cite{Wal80}. On the other hand, for non-persistent covers, the equality \eqref{E:main2} no longer holds, and the quantitative nature of $c_\mca{O}$ has to be investigated further.
\end{enumerate}

The first example for Conjecture \ref{C:main} is when $n=1$ and thus $\wt{G}=G$ is a linear group. In this case, $\Theta(\pi^\dag, \nu)$ is a one-dimensional character of $G$. Since $G$ is saturated, we can take $\tnu = \nu$. On the other hand, one has $\Phi_\nu =\Phi$ and thus $j_{W_\nu}^W(\varepsilon_\nu) = \varepsilon_W$. This gives 
$$\mca{O}_{\rm Spr}(j_{W_\nu}^W(\varepsilon_\nu)) = \mca{O}_0,$$
as expected. In this case, $\msc{X}_{Q,n}=\set{0}$ and thus $\sigma^\msc{X} = \mbm{1}$. It follows 
$$c_{\mca{O}_0} = 1 =\angb{\varepsilon_W}{ \varepsilon_W \otimes \mbm{1}}_W = \angb{j_{W_\tnu}^W(\varepsilon_\tnu)}{ \varepsilon_W \otimes \sigma^\msc{X} }_W.$$

As another extreme example, we take an unramified $\Theta(\pi^\dag, \nu)$ with $n \gg r$, where $r$ is the semisimple rank of $G$. In this case, $\Theta(\pi^\dag, \nu)$ is $\psi$-generic for $\mfr{f}(\psi)=O_F$. In fact, it follows from \cite[Proposition 6.2]{Ga6} that
the equality
\begin{equation} \label{E:dimWh}
\dim \Wh_\psi(\Theta(\pi^\dag, \nu)) = \angb{\varepsilon_W}{ \sigma^\msc{X}}_W
\end{equation}
holds for every persistent cover. Here $\angb{\varepsilon_W}{ \sigma^\msc{X}}_W$ also equals to the 
number of free $W$-orbits in $\msc{X}_{Q,n}$ with respect to the twisted action $w[-]$.  We then have a regular $F$-nilpotent orbit $\mca{O}$ (depending on both $\psi_\natural$ and $\psi$, see \cite[Page 427]{MW1}) that
$$c_{\mca{O}} = \dim \Wh_\psi(\Theta(\pi^\dag, \nu)).$$
In this case, $\phi(\tnu) \in X^{sc}\otimes \R$ lies in the interior of the alcove with respect to the affine Weyl group $X^{sc} \rtimes W$ acting on $X^{sc}\otimes \R$. Thus $W_\tnu=\set{1}$ and $j_{W_\tnu}^W(\varepsilon_\tnu) = \mbm{1}$. This shows that
$$c_\mca{O} =\angb{\varepsilon_W}{\sigma^\msc{X}}_W =  \angb{j_{W_\tnu}^W(\varepsilon_\tnu)}{ \varepsilon_W \otimes \sigma^\msc{X} }_W,$$ as desired.  In \S \ref{S:generic}, we will study in more details the case when $\Theta(\pi^\dag, \nu)$ is $\psi$-generic and verify \eqref{E:main1} and \eqref{E:main2} for $\mca{O}$ as above.

\subsection{A further generalization}
We briefly discuss about a further generalization of Conjecture \ref{C:main} to all irreducible constituents of a regular principal series in the unramified case. Thus, we continue to assume that $F$ is $p$-adic with $p\nmid n$. Consider $\nu \in X\otimes \R$ satisfying the following:
\begin{enumerate}
\item[--] $\nu$ is regular, that is, its stabilizer subgroup of $W$ is trivial,
\item[--] the set $\Phi(\nu):=\set{\alpha\in \Phi: \nu(\alpha_{Q,n}^\vee) =1}$ is a subset of $\Delta$.
\end{enumerate}
Taking $\pi^\dag$ to be an unramified distinguished representation of $\wt{T}$, we have a regular unramified genuine principal series $I(\pi^\dag, \nu)$ of $\wt{G}$.
One has
\begin{equation} \label{E:Rod}
I(\pi^\dag, \nu)^{\rm ss} = \bigoplus_{S \subset \Phi(\nu)} \pi_S,
\end{equation}
where the left hand side denotes the semisimplification of $I(\pi^\dag, \nu)$. The decomposition is multiplicity-free and the irreducible constituent $\pi_S$ is characterized by its Jacquet module, see \cite{Rod4} and \cite[\S 3]{Ga6}. For example, if $\Phi(\nu) =\Delta$, then $\pi_\Delta = \Theta(\pi^\dag, \nu)$ and $\pi_\emptyset$ is a covering analogue of the Steinberg representation.

For every $S \subset \Phi(\nu) \subset \Delta$, let $\Phi(S) \subset \Phi$ be the root subsystem with simple roots being $S$. Denote by
$$W(S) \subset W$$
the subgroup generated by elements in $S$. Let $M_S \subset G$ be the Levi subgroup associated with $S$, with Lie algebra denoted by $\mfr{m}_S$. Let $\tnu \in X\otimes \R$ be a saturation of $\nu$ and denote
$$W_\tnu^S:=\text{the integral Weyl subgroup of $W(S)$ associated with }\tnu.$$
Let $\varepsilon_\tnu^S$ be the sign character of $W_\tnu^S$. We have the two Macdonald representations
$$j_{W_\tnu^S}^W(\varepsilon_\tnu^S) \in \Irr(W) \text{ and } \ j_{W_\tnu^S}^{W(S)}(\varepsilon_\tnu^S) \in \Irr(W(S)).$$
For every nilpotent orbit $\mca{O} \subset \mfr{m}_S\otimes \wt{F}$, one has an induced nilpotent orbit ${\rm Ind}_{\mfr{m}_S\otimes \wt{F}}^{\mfr{g} \otimes \wt{F}}(\mca{O}) \subset \mfr{g} \otimes \wt{F}$, see \cite{LuSp1} or \cite[Chapter 7]{CM}. Moreover, the $j$-induction on the representation side and induction on the nilpotent orbit side from parabolic subgroups are compatible via the Springer correspondence, see \cite{LuSp1}. Thus, we have
$$\mca{O}_{\rm Spr}(j_{W_\tnu^S}^W(\varepsilon_\tnu^S)) ={\rm Ind}_{\mfr{m}_S\otimes \wt{F}}^{\mfr{g} \otimes \wt{F}}( \mca{O}_{\rm Spr}(j_{W_\tnu^S}^{W(S)}(\varepsilon_\tnu^S))) \subset \mfr{g}\otimes \wt{F}.$$

\begin{conj} \label{C:main-g} 
Let $\wt{G}$ be a persistent $n$-fold cover, and let $\nu\in X\otimes \R$ be a regular element with $\Phi(\nu) \subset \Delta$. Consider the regular unramified principal series $I(\pi^\dag, \nu)$.  Then for every constituent $\pi_S$ of $I(\pi^\dag, \nu)$ with $S \subset \Phi(\nu)$, one has
$$\mca{N}_{\rm tr}^{\rm max}(\pi_S) \otimes \wt{F} = \set{ \mca{O}_{\rm Spr}(j_{W_\tnu^S}^W(\varepsilon_\tnu^S)) },$$
where $\tnu$ is a saturation of $\nu$.
\end{conj}
If $S=\Phi(\nu) =\Delta$, then Conjecture \ref{C:main-g} becomes part of Conjecture \ref{C:main}, since $\pi_\Delta =\Theta(\pi^\dag, \nu)$ in this case. On the other hand, if $n=1$ and $\wt{G}=G$ is linear group, then Conjecture \ref{C:main-g} asserts that
$$\mca{N}_{\rm tr}^{\rm max}(\pi_S) \otimes \wt{F} =\set{{\rm Ind}_{\mfr{m}_S\otimes \wt{F}}^{\mfr{g}\otimes \wt{F}}(0)},$$
which was proved by M\oe glin and Waldspurger in \cite[Proposition II.1.3]{MW1}.

\begin{rmk} Conjecture \ref{C:main} was stated for $\Theta(\pi^\dag, \nu)$ in the tame case, but with $\pi^\dag$ not necessarily unramified. On the other hand, we restrict to unramified $\pi^\dag$ in Conjecture \ref{C:main-g}, since Rodier's structural decomposition \eqref{E:Rod} of $I(\pi^\dag, \nu)$ was analyzed and generalized only for unramified data in \cite{Ga6}. However, it is expected that $\eqref{E:Rod}$ holds for regular $I(\pi^\dag, \nu)$ in general (even in the non-tame setting) which, once established, will enable us to remove the constraint of $\pi^\dag$ being unramified in Conjecture \ref{C:main-g}.
\end{rmk}

%%%
\section{Generic $\Theta(\pi^\dag, \nu)$} \label{S:generic}
In this section, we show that certain parts of Conjecture \ref{C:main} hold for a $\psi$-generic $\Theta(\pi^\dag, \nu)$ in the unramified case. We assume 
$$\mfr{f}(\psi) = O_F.$$
Essentially, we rely on the results proved in \cite{Ga6} regarding the criterion for $\Theta(\pi^\dag, \nu)$ to be generic. More precisely, it follows from \eqref{E:dimWh} for a persistent cover $\wt{G}$ that the following two assertions are equivalent:
\begin{enumerate}
\item[(i)] the representation $\Theta(\pi^\dag, \nu)$ is $\psi$-generic, and thus $\mca{N}_{\rm tr}^{\rm max}(\Theta(\pi^\dag, \nu)) \otimes \wt{F} = \set{\mca{O}_{\rm reg}}$,
\item[(ii)] the number $\angb{\varepsilon_W}{ \sigma^\msc{X} }_W$, which is equal to the number of free $W$-orbits in $\msc{X}_{Q,n}$, is at least one.
\end{enumerate}
We note that, however, property (ii) here concerns $Y/Y_{Q,n}$ on the cocharacter lattice side, while Conjecture \ref{C:main} relies on the element $\tnu \in X\otimes \R$ from the character lattice side. It is thus sufficient to establish the ``equivalence" between the two criteria for $\Theta(\pi^\dag, \nu)$ to be generic, arising from (ii) above and that predicted by Conjecture \ref{C:main}.

Recall that $\set{\omega_\alpha}_{\alpha\in \Delta} \subset X\otimes \R$ denote the fundamental weights, and $\rho = \sum_{\alpha\in \Delta} \omega_\alpha$. Similarly, $\set{\omega_\alpha^\vee}_{\alpha \in \Delta} \subset Y\otimes \R$ denote the fundamental coweights with $\rho^\vee=\sum_{\alpha\in \Delta} \omega_\alpha^\vee$. Define a function 
$$f_X: \Phi_+^\vee \longrightarrow \Q$$
by 
$$f_X(\beta^\vee) :=\angb{\phi(\tnu)}{\beta^\vee}= \sum_{\alpha\in \Delta}\angb{\omega_\alpha/\tilde{n}_\alpha}{\beta^\vee}.$$
It is clear that $W_\tnu = \set{1}$ if and only if ${\rm Im}(f_X) \cap \Z =\emptyset$.
On the other hand, we also define
$$f_Y: \Phi_+ \longrightarrow \Q$$
by 
$$f_Y(\beta) = \angb{\rho^\vee}{\tilde{\beta}_{Q,n}} = \sum_{\alpha\in \Delta} \angb{\omega_\alpha^\vee/\tilde{n}_\beta}{\beta}.$$
The $W$-orbit of $2\rho^\vee$ in $\msc{X}_{Q,n}$ is free if and only if  ${\rm Im}(f_Y) \cap \Z =\emptyset$.

We will prove the equality
$$f_X(\Phi_+^\vee) = f_Y(\Phi_+)$$ 
on a case-by-case basis for the root system type of $G$. The analysis will also determine explicitly some saturated and persistent covers.

\subsection{Simply-laced type with rank $r\ge 2$} It follows from Lemma \ref{L:clas} that $\tilde{n}_\alpha = n_\alpha$ for all $\alpha \in \Delta$, and since $G$ is simply-laced, the map $\alpha \mapsto n_\alpha$ is constant on $\Delta$. Also, since the root system $\Phi^\vee$ is of the same type as $\Phi$, we have the following.

\begin{lm} \label{L:sim-lac}
If the root system of $G$ is of simply-laced type, then
$${\rm Im}(f_X) = \frac{[1, {\rm ht}(\alpha_0)]}{n_\alpha} = {\rm Im}(f_Y),$$
where ${\rm ht}(\alpha_0)$ denotes the height of the highest root $\alpha_0$ of $G$.
\end{lm}
We discuss briefly the saturated and persistent covers $\wt{G}^{(n)}$ of the almost simple group $G$ associated with 
$$Q(\alpha^\vee)=1,$$
where $\alpha^\vee \in \Delta^\vee$ is any simple coroot. For simplicity, we may sometimes restrict to the case $n=2$ only.

First, all saturated covers of such $\wt{G}^{(n)}$ are given in Table 1 (see \cite[\S 2.7]{We6})
\begin{table}[!htbp]  \label{T:1}
\caption{Saturated covers for simply-laced almost simple groups}
\vskip 5pt
\begin{tabular}{|c|c|c|c|c|c|c|}
\hline
 & $\wt{\SL}_{r+1}^{(n)}$  &  $\wt{\Spin}_{2r}^{(n)}, r\gest 3$ &     $\wt{E}_6^{(n)}$  &  $\wt{E}_7^{(n)}$ & $\wt{E}_8^{(n)}$  \\
\hline
condition & $n$ \text{ such that} & odd $n$  & \text{$n$ such that } & odd $n$  & all $n$   \\
on $n$ & $\gcd(n, r+1)=1$  & & $3\nmid n$ &  &  \\
\hline
%$ \val{Y^\exc_n}=1 $ & all $n$  & odd $n$  & all $n$ & all $n$ & all $n$ \\
% & & & & & \\
% \hline
\end{tabular}
\end{table}

We also tabulate the persistent double covers $\wt{G}^{(2)}$, as given in Table 2. Recall if a cover is saturated, then it is persistent. Now we explain briefly the part of Table 2 not covered in Table 1 as follows. Thus, we assume $\wt{G}$ is either $\wt{\SL}_{r+1}^{(2)}$ with $r$ odd, or $\wt{\Spin}_{2r}^{(2)}$ and $\wt{E}_7^{(2)}$. Let $\Omega \subset \Delta$ be the special subset as in \cite[\S 16.1.1]{GG} such that 
$$e_\Omega = \sum_{\alpha \in \Omega} \alpha^\vee \in Y_{Q,2} - Y_{Q,2}^{sc}.$$
More precisely, using the labelling as in Bourbaki \cite{Bou} (which is however different from that in \cite[\S 16.1.1]{GG}), if $\wt{G}=\wt{\SL}_{r+1}^{(2)}$, then we take $\Omega=\set{\alpha_i: 1\lest i \lest r, i \text{ is odd}}$; for $\wt{G}= \wt{\Spin}_{2r}^{(2)}$, we take $\Omega=\set{\alpha_{r-1}, \alpha_r}$; while for $\wt{G}=\wt{E}_7^{(2)}$, we take $\Omega=\set{\alpha_2, \alpha_5, \alpha_7}$.
Setting $w_\Omega = \prod_{\alpha \in \Omega} w_\alpha$, which does not depend on the order of elements in $\Omega$, we have
$$w_\Omega[0] - 0 = e_\Omega \in Y_{Q,2} - Y_{Q,2}^{sc}.$$
This shows that such $\wt{G}$ is not persistent. This justifies Table 2.

\begin{table}[!htbp]  \label{T:2}
\caption{Double covers of simply-laced almost simple groups}
\vskip 5pt
\begin{tabular}{|c|c|c|c|c|c|c|c|}
\hline
 & $\wt{\SL}_{r+1}^{(2)},$  & $\wt{\SL}_{r+1}^{(2)},$ &  $\wt{\Spin}_{2r}^{(2)}$ &     $\wt{E}_6^{(2)}$  &  $\wt{E}_7^{(2)}$ & $\wt{E}_8^{(2)}$  \\
 & $r$ even  & $r$ odd &   &     &  &   \\
\hline
saturated or & saturated & not & not   & saturated & not  & saturated   \\
persistent? &   & persistent & persistent  &  & persistent  &  \\
\hline
%$ \val{Y^\exc_n}=1 $ & all $n$  & odd $n$  & all $n$ & all $n$ & all $n$ \\
% & & & & & \\
% \hline
\end{tabular}
\end{table}

\subsection{Type $B_r, r\gest 3$} \label{SS:Br}
Following Bourbaki's notation \cite[Page 267]{Bou}, we have the Dynkin diagram for the root system of type $B_r, r\gest 3$ as follows.
$$ \qquad
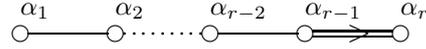
\begin{picture}(4.7,0.2)(0,0)
\put(1,0){\circle{0.08}}
\put(1.5,0){\circle{0.08}}
\put(2,0){\circle{0.08}}
\put(2.5,0){\circle{0.08}}
\put(3,0){\circle{0.08}}
\put(1.04,0){\line(1,0){0.42}}
\multiput(1.55,0)(0.05,0){9}{\circle*{0.02}}
\put(2.04,0){\line(1,0){0.42}}
\put(2.54,0.015){\line(1,0){0.42}}
\put(2.54,-0.015){\line(1,0){0.42}}
\put(2.72,-0.04){$>$}
\put(1,0.1){\footnotesize $\alpha_1$}
\put(1.5,0.1){\footnotesize $\alpha_2$}
\put(2,0.1){\footnotesize $\alpha_{r-2}$}
\put(2.5,0.1){\footnotesize $\alpha_{r-1}$}
\put(3,0.1){\footnotesize $\alpha_r$}
\end{picture}
$$
\vskip 10pt
We partition $\Phi_+$ into 
$$\Phi_+ = \Phi_{\rm +, I} \sqcup \Phi_{\rm +, II} \sqcup \Phi_{\rm +, III},$$
where 
$$\Phi_{\rm +, I} = \set{\sum_{i\lest k \lest r} \alpha_k: 1\lest i \lest r}, \ \Phi_{\rm +, II} = \set{\sum_{i\lest k <j} \alpha_k : 1\lest i < j \lest r} $$
and
$$\Phi_{\rm +, III} =\set{\sum_{i\lest k <j} \alpha_k + 2\cdot \sum_{j\lest k \lest r} \alpha_k: 1\lest i < j \lest r}.$$
On the other hand, one can partition $\Phi_+^\vee$ into 
$$\Phi_+^\vee = \Phi_{\rm +, I}^\vee \sqcup \Phi_{\rm +, II}^\vee \sqcup \Phi_{\rm +, III}^\vee,$$
where 
$$\Phi_{\rm +, I}^\vee =  \set{\sum_{i\lest k <j} \alpha_k^\vee : 1\lest i < j \lest r} , \ \Phi_{\rm +, III}^\vee =\set{2\sum_{i\lest k < r} \alpha_k^\vee + \alpha_r^\vee: 1\lest i \lest r}$$
and
$$\Phi_{\rm +, II}^\vee = \set{\sum_{i\lest k <j} \alpha_k^\vee + 2\cdot \sum_{j\lest k < r} \alpha_k^\vee + \alpha_r^\vee: 1\lest i < j \lest r}.$$
There are two cases for the set $\set{n_\alpha, \alpha \in \Delta}$: either $n_{\alpha}$ is constant on $\alpha\in \Delta$, or $n_{\alpha_i} =2 n_{\alpha_r}$ for all $1\lest i < r$.

\begin{lm} \label{L:Br}
Assume the root system of $G$ is of type $B_r, r\gest 3$. If $n_{\alpha_i} = 2n_{\alpha_r}$ for all $1\lest i<r$, then 
$$
f_X(\Phi_{\rm +, I}^\vee) = f_Y(\Phi_{\rm +, II}), \ f_X(\Phi_{\rm +, II}^\vee) = f_Y(\Phi_{\rm +, III}), \ f_X(\Phi_{\rm +, III}^\vee) = f_Y(\Phi_{\rm +, I});
$$
in particular, 
$$f_X(\Phi_+^\vee) =\frac{1}{2n_{\alpha_r}} [1, 2r] =  f_Y(\Phi_+).$$
If $n_\alpha$ is constant on $\alpha\in \Delta$, then we have the equalities
$$f_X(\Phi_+^\vee) = \frac{1}{n_\alpha} [1, 2r-1] = f_Y(\Phi_+).$$
\end{lm}
\begin{proof}
By assumption on the root system of $G$, we have $\tilde{n}_\alpha = n_\alpha$ for all $\alpha \in \Delta$, see Lemma \ref{L:clas}.
The above equalities  then follow from an easy computation. We omit the details. 
\end{proof}
Note that Lemma \ref{L:Br} clearly agrees with the discussion in \cite[\S 6]{Ga2} for $G=\Spin_{2r+1}$. 

Let $\wt{\Spin}_{2r+1}^{(n)}$ be the $n$-fold cover associated with $Q(\alpha_1^\vee)=1$. If $r$ is even and $n=2k$ with $k$ odd, then one can check easily that (writing $w_i = w_{\alpha_i}$)
$$w_{r-1} w_{r-3} ... w_3 w_1[(1-k)\rho^\vee] - (1-k)\rho^\vee = k\cdot (\alpha_1^\vee + \alpha_3^\vee + ... + \alpha_{r-1}^\vee) \in Y_{Q,n} - Y_{Q,n}^{sc},$$
where $(1-k)\rho \in Y$. This shows that such $\wt{\Spin}_{2r+1}^{(n)}$ is not persistent. This gives us Table 3. We remark that for $n\in 4\Z$, the persistence for $\wt{\Spin}_{2r+1}^{(n)}$ does not depend solely on the parity of $r$ and is quite complicated.

\begin{table}[!htbp]  \label{T:3}
\caption{Covers of $\Spin_{2r+1}$}
\vskip 5pt
\begin{tabular}{|c|c|c|c|}
\hline
 & $\wt{\Spin}_{2r+1}^{(n)}, r$ odd  &  $\wt{\Spin}_{2r+1}^{(n)}, r$ even   \\
\hline
$n$ is odd & saturated & saturated    \\
\hline
$n \in 4\Z +2$ & saturated  &  not persistent  \\
%\hline
% $n\in 4\Z$ &   &   \\
\hline
%$ \val{Y^\exc_n}=1 $ & all $n$  & odd $n$  & all $n$ & all $n$ & all $n$ \\
% & & & & & \\
% \hline
\end{tabular}
\end{table}

\subsection{Type $C_r$} \label{SS:Sp}
Consider the Dynkin diagram for the root system of type $C_r, r\gest 2$ as follows.
$$ \qquad
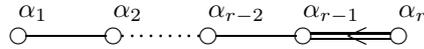
\begin{picture}(4.7,0.2)(0,0)
\put(1,0){\circle{0.08}}
\put(1.5,0){\circle{0.08}}
\put(2,0){\circle{0.08}}
\put(2.5,0){\circle{0.08}}
\put(3,0){\circle{0.08}}
\put(1.04,0){\line(1,0){0.42}}
\multiput(1.55,0)(0.05,0){9}{\circle*{0.02}}
\put(2.04,0){\line(1,0){0.42}}
\put(2.54,0.015){\line(1,0){0.42}}
\put(2.54,-0.015){\line(1,0){0.42}}
\put(2.72,-0.04){$<$}
\put(1,0.1){\footnotesize $\alpha_1$}
\put(1.5,0.1){\footnotesize $\alpha_2$}
\put(2,0.1){\footnotesize $\alpha_{r-2}$}
\put(2.5,0.1){\footnotesize $\alpha_{r-1}$}
\put(3,0.1){\footnotesize $\alpha_r$}
\end{picture}
$$
\vskip 10pt
We can also partition $\Phi_+$ and $\Phi_+^\vee$ as follows, which is dual to that in the type $B_r$ case. First,
$$\Phi_+ = \Phi_{\rm +, I} \sqcup \Phi_{\rm +, II} \sqcup \Phi_{\rm +, III},$$
where 
$$\Phi_{\rm +, I} =  \set{\sum_{i\lest k <j} \alpha_k : 1\lest i < j \lest r} , \ \Phi_{\rm +, III} =\set{2\sum_{i\lest k < r} \alpha_k + \alpha_r: 1\lest i \lest r}$$
and
$$\Phi_{\rm +, II} = \set{\sum_{i\lest k <j} \alpha_k + 2\cdot \sum_{j\lest k < r} \alpha_k + \alpha_r: 1\lest i < j \lest r}.$$
On the other hand, we have
$$\Phi_+^\vee = \Phi_{\rm +, I}^\vee \sqcup \Phi_{\rm +, II}^\vee \sqcup \Phi_{\rm +, III}^\vee,$$
where
$$\Phi_{\rm +, I}^\vee = \set{\sum_{i\lest k \lest r} \alpha_k^\vee: 1\lest i \lest r}, \ \Phi_{\rm +, II}^\vee = \set{\sum_{i\lest k <j} \alpha_k^\vee : 1\lest i < j \lest r} $$
and
$$\Phi_{\rm +, III}^\vee =\set{\sum_{i\lest k <j} \alpha_k^\vee + 2\cdot \sum_{j\lest k \lest r} \alpha_k^\vee: 1\lest i < j \lest r}.$$

Again, there are two cases for the set $\set{\tilde{n}_\alpha, \alpha \in \Delta}$: either $\tilde{n}_{\alpha}$ is constant on $\alpha\in \Delta$, or $2\tilde{n}_{\alpha_i} =\tilde{n}_{\alpha_r}$ for all $1\lest i < r$.

\begin{lm} \label{L:Cr}
Assume the root system of $G$ is of type $C_r, r\gest 2$. If $2\tilde{n}_{\alpha_i} = \tilde{n}_{\alpha_r}$ for all $1\lest i<r$, then 
necessarily $\tilde{n}_{\alpha_i} = n_{\alpha_i}$ for all $1\lest i \lest r$ and
$$
f_X(\Phi_{\rm +, I}^\vee) = f_Y(\Phi_{\rm +, III}), \ f_X(\Phi_{\rm +, II}^\vee) = f_Y(\Phi_{\rm +, I}), \ f_X(\Phi_{\rm +, III}^\vee) = f_Y(\Phi_{\rm +, II});
$$
in this case,
$$f_X(\Phi_+^\vee) =\frac{[1, 2r-1]}{2n_{\alpha_1}} \bigcup \frac{[2, 2r-2]}{n_{\alpha_1}} =  f_Y(\Phi_+).$$
If $\tilde{n}_\alpha$ is constant on $\alpha\in \Delta$, then we have the equalities
$$f_X(\Phi_+^\vee) = \frac{1}{\tilde{n}_\alpha} [1, 2r-1] = f_Y(\Phi_+).$$
\end{lm}
\begin{proof}
If $n_{\alpha_i} \ne \tilde{n}_{\alpha_i}$, then it follows from Lemma \ref{L:clas} that $i=r$ and that $\tilde{n}_{\alpha_r} = \tilde{n}_{\alpha_i}$ for all $1\lest i <r$. The rest of the assertions follows from an easy computation as for Lemma \ref{L:Br}.
\end{proof}
It is clear that Lemma \ref{L:Cr} agrees with the pertinent discussion in \cite[\S 5]{Ga2} for covers of $\Sp_{2r}$.

Let $\wt{\Sp}_{2r}^{(n)}$ be the $n$-fold cover associated with $Q(\alpha_r^\vee)=1$. We determine when $\wt{\Sp}_{2r}^{(n)}$ is saturated or persistent. If $n$ is odd, then the dual group of $\wt{\Sp}_{2r}^{(n)}$ is $\SO_{2r+1}$, and thus $\wt{\Sp}_{2r}^{(n)}$ is saturated. For even $n$, the cover $\wt{\Sp}_{2r}^{(n)}$ is not saturated, since its dual group is $\Sp_{2r}$.

If $n=2(2k -1) \in 4\Z +2$, then it is easy to see that 
$$w_{\alpha_r}[k\alpha_r^\vee] - k\alpha_r^\vee = (1-2k)\alpha_r^\vee  \in Y_{Q,n} - Y_{Q,n}^{sc}.$$
This shows that $\wt{\Sp}_{2r}^{(n)}$ is not persistent for $n\in 4\Z+2$. For $n\in 4\Z$ we have the following.

\begin{lm}  \label{L:4n-pers}
If $4|n$, then $\wt{\Sp}_{2r}^{(n)}$ is a persistent covering group.
\end{lm}
\begin{proof}
By the definition of persistence, it suffices to show that for every $y\in Y$, one has
$${\rm Stab}_W(y, \msc{X}_{Q, n}^{sc}) = {\rm Stab}_W(y, \msc{X}_{Q,n}),$$
where the inclusion $\subseteq$ is clear. We write $n=2m$ with $m$ even. Then
$$Y_{Q,n}^{sc} = \set{\sum_{i=1}^r c_i \alpha_i^\vee: c_i \in m\Z \text{ for } i\ne r, \text{ and } c_r \in n\Z}$$
and 
$$Y_{Q,n}= \set{\sum_{i=1}^r c_i \alpha_i^\vee: c_i \in m\Z \text{ for every }i}.$$
If ${\rm Stab}_W(y, \msc{X}_{Q,n}^{sc}) \ne {\rm Stab}_W(y, \msc{X}_{Q,n})$, then there exists $z\in \mca{O}_y$ (the $W$-orbit of $y$ in $\msc{X}_{Q,n}$) such that $z_\rho:=z-\rho^\vee$ lies in the hyperplane $H_{\alpha_r} \subset Y\otimes_\Z \R$ associated to the affine Weyl element $(m\alpha_r^\vee, w_{\alpha_r}) \in Y_{Q,n} \rtimes W$. That is, $w_{\alpha_r}$ fixes $z_\rho - \frac{m}{2}\alpha_r^\vee$. We thus obtain
$$m=\angb{z_\rho}{\alpha_r},$$
the right hand side of which however is always an odd number. This gives a contradiction. Thus every such $\wt{\Sp}_{2r}$ is persistent.
\end{proof}
As a conclusion from the above discussion, we have Table 4 below.

\begin{table}[!htbp]  \label{T:4}
\caption{Covers of $\Sp_{2r}$}
\vskip 5pt
\begin{tabular}{|c|c|c|c|c|}
\hline
 & $n$ is odd   &  $n\in 4\Z + 2$ & $n\in 4\Z$   \\
\hline
$\wt{\Sp}_{2r}^{(n)}$ & saturated & not persistent  & persistent, but    \\
 &   &  & not saturated\\
%\hline
% $n\in 4\Z$ &   &   \\
\hline
%$ \val{Y^\exc_n}=1 $ & all $n$  & odd $n$  & all $n$ & all $n$ & all $n$ \\
% & & & & & \\
% \hline
\end{tabular}
\end{table}

\subsection{Type $F_4$}
Consider the Dynkin diagram of simple roots of $F_4$ as follows
$$
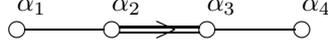
\begin{picture}(4.7,0.2)(0,0)
\put(2,0){\circle{0.08}}
\put(2.5,0){\circle{0.08}}
\put(3,0){\circle{0.08}}
\put(3.5,0){\circle{0.08}}
\put(2.04,0){\line(1,0){0.42}}
\put(2.54,0.015){\line(1,0){0.42}}
\put(2.54,-0.015){\line(1,0){0.42}}
\put(2.72,-0.04){$>$}
\put(3.04,0){\line(1,0){0.42}}
\put(2,0.1){\footnotesize $\alpha_1$}
\put(2.5,0.1){\footnotesize $\alpha_2$}
\put(3,0.1){\footnotesize $\alpha_3$}
\put(3.5,0.1){\footnotesize $\alpha_4$}
\end{picture}
$$

\vskip 10pt 
Note that covers of $F_4$ are always saturated. An explicit computation using the data and notations in \cite[Page 287-288]{Bou} gives the following.
\begin{lm} \label{L:F4}
For any $n$-fold cover of $F_4$, if $n_{\alpha_1} = 2n_{\alpha_4}$, then one has
$$f_X(\Phi_+^\vee) = f_Y(\Phi_+) = \frac{1}{n_{\alpha_4}} [1, 8] \bigcup \frac{1}{2n_{\alpha_4}} \set{1, 3, 5, 7, 9, 11}.$$
If $n_{\alpha}$ is constant on $\alpha \in \Delta$, then 
$$f_X(\Phi_+^\vee) = f_Y(\Phi_+) = \frac{1}{n_\alpha} [1, 11].$$
\end{lm}

\subsection{Type $G_2$}

Consider the Dynkin diagram of $G_2$:
$$
\begin{picture}(5.2,0.2)(0,0)
\put(2.5,0){\circle{0.08}}
\put(3,0){\circle{0.08}}
\put(2.53,0.018){\line(1,0){0.44}}
\put(2.54,0){\line(1,0){0.42}}
\put(2.53,-0.018){\line(1,0){0.44}}
\put(2.7,-0.04){$<$}
\put(2.5,0.1){\footnotesize $\alpha_1$}
\put(3,0.1){\footnotesize $\alpha_2$}
\end{picture}
$$

\vskip 10pt 
Every $n$-fold cover of $G_2$ is saturated, since the dual group is always $G_2$.
\begin{lm} \label{L:G2}
For $n$-fold cover of $G_2$, if $n_{\alpha_2} = 3 n_{\alpha_1}$, then 
$$f_X(\Phi_+^\vee) = f_Y(\Phi_+) =\frac{1}{3n_{\alpha_1}} \set{1, 4, 5} \bigcup \frac{1}{n_{\alpha_1}} \set{1, 2, 3}  \subset \Q.$$
If $n_{\alpha_1} = n_{\alpha_2}$, then 
$$f_X(\Phi_+^\vee) = f_Y(\Phi_+) = \frac{1}{n_\alpha} [1, 5].$$
\end{lm}

\subsection{Generic $\Theta(\pi^\dag, \nu)$}
Using the proceeding discussion, we have the following result for $\psi$-generic $\Theta(\pi^\dag, \nu)$.

\begin{thm}  \label{T:generic}
Let $\wt{G}$ be a $p$-adic persistent cover  in the tame case. Let $\Theta(\pi^\dag, \nu)$ be an unramified theta representation, and let $\tnu \in X\otimes \R$ be a saturation of the exceptional $\nu$. Then the following are equivalent:
\begin{enumerate}
\item[(i)] $W_{\tnu} = \set{1}$,
\item[(ii)] the $W$-orbit of $2\rho^\vee$ in $\msc{X}_{Q,n}$ with respect to $w[-]$ is free,
\item[(iii)] the $W$-orbit of $0$ in $\msc{X}_{Q,n}$ with respect to $w[-]$ is free,
\item[(iv)] $\angb{\varepsilon_W}{ \sigma^\msc{X}}_W \gest 1$, i.e., there is at least one $W$-free orbit in $\msc{X}_{Q,n}$.
\end{enumerate}
If $\Theta(\pi^\dag, \nu)$ is $\psi$-generic, then by setting $\mca{O} \subset \mfr{g}$ to be the regular nilpotent (depending on $\psi_\natural$ and $\psi$) such that $c_\mca{O} = \dim \Wh_\psi(\Theta(\pi^\dag, \nu))$, we have
$$c_\mca{O} =  \angb{j_{W_\tnu}^W(\varepsilon_\tnu)}{ \varepsilon_W \otimes \sigma^\msc{X} }_W.$$
That is, \eqref{E:main1} and \eqref{E:main2} for this $\mca{O}$ in Conjecture \ref{C:main} both hold for the $\psi$-generic $\Theta(\pi^\dag, \nu)$.
\end{thm}
\begin{proof}
First, we show the equivalence between (i) and (ii). As noted in the beginning of this section that (i) and (ii) are equivalent to the equalities $f_X(\Phi_+^\vee) = \emptyset$ and $f_Y(\Phi_+) =\emptyset$, respectively. However, if follows from Lemmas \ref{L:sim-lac}, \ref{L:Br}, \ref{L:Cr}, \ref{L:F4} and \ref{L:G2} that 
$$f_X(\Phi_+^\vee) = f_Y(\Phi_+)$$
holds for all root system types. Thus the equivalence of (i) and (ii) follows.

The equivalence between (ii) and (iii) is trivial. It is also clear that (ii) implies (iv); thus, it suffices to show the converse. 
Write $\Delta=\set{\alpha_i: 1\lest i \lest r}$. Every $y\in Y$ can be written in the form
$$y= y_0 + \sum_{\alpha_i\in \Delta} c_i \alpha_i^\vee$$
where $y_0\in Y\otimes \Q$ satisfies $\angb{y_0}{\alpha_i}=0$ for every $\alpha_i \in \Delta$, and $c_i \in \Q$, see \cite[Lemma 1.2]{Spr1}. In particular, $y_0$ is fixed by the Weyl group. If we write
$$y - y_0 = \sum_{i} k_i \omega_i^\vee \in Y^{sc}\otimes \R$$
in terms of the basis $\set{\omega_i^\vee}$ for $Y^{sc}\otimes \R$, then $k_i = \angb{y-y_0}{\alpha_i} = \angb{y}{\alpha_i} \in \Z$.
Now let 
$$\tilde{\omega}_i^\vee= \tilde{n}_{\alpha_i} \cdot \omega_i^\vee \in Y^{sc}\otimes \R$$
 be the fundamental coweight associated with $\tilde{\alpha}_{Q,n}^\vee$. Consider the affine Weyl group $\tilde{W} = \tilde{Y}_{Q,n}^{sc} \rtimes W$ acting on $Y^{sc}\otimes \R$.
Let $C \subset Y^{sc}\otimes \R$ be a fundamental alcove with extreme points ${0} \cup \set{\tilde{\omega}_i^\vee/g_i: 1\lest i \lest r}$ with $g_i \in \N$, see \cite[Page 187-188]{Bou}. Assertion (iv) is equivalent to that there exists $y\in Y$ such that
$$ y - y_0 -\rho^\vee = k_1 \omega_1^\vee + k_2 \omega_2^\vee + ... + k_r \omega_r^\vee \in  C,$$
where $k_i \in \Z_{\gest 0}$ for every $i$. We have
$$y - y_0 - \rho^\vee =\sum_{i=1}^r \frac{k_ig_i}{\tilde{n}_i}(\tilde{\omega}_i^\vee/g_i).$$
Since $y - y_0 - \rho^\vee$ lies in $C$, it gives that $k_i g_i >0$ for every $i$ and moreover
$$\sum_{i=1}^r  \frac{k_ig_i}{\tilde{n}_i} < 1.$$
Thus, $k_i \gest 1$ for all i, and one has $\sum_{i=1}^r  g_i/\tilde{n}_i < 1$. However, this shows that $\rho^\vee = \sum_i \omega_i^\vee$ lies in $C$ and therefore the $W$-orbit of $2\rho^\vee \in Y$ in $\msc{X}_{Q,n}$ is free. Hence, (iv) implies (ii).

The last assertion follows from \eqref{E:dimWh} and the equivalence between (i) and (iv). The proof is completed.
\end{proof}

\begin{rmk}
Since we have an explicit form for the set $f_X(\Phi_+^\vee) = f_Y(\Phi_+)$ for all root system types, for persistent covers $\wt{G}^{(n)}$ one can determine precisely the mininum $n$ such that $\Theta(\pi^\dag, \nu)$ is generic. It is also possible to determine $n$ such that $\Theta(\pi^\dag, \nu)$ is distinguished, i.e., $\dim \Wh_\psi(\Theta(\pi^\dag, \nu))=1$. The results agree with \cite{Ga2} for covers of type $A_r, B_r, C_r$ and $G_2$ discussed there. For instance, as an example not covered by \cite{Ga2}, it follows from Lemma \ref{L:sim-lac} that for the $n$-fold cover of $E_8$ with $Q(\alpha^\vee)=-1$, one has
$$\dim \Wh_\psi(\Theta(\pi^\dag, \nu)) \gest 1$$
exactly when $n\gest 30$, and the equality holds for $n=30$.
\end{rmk}

%%%

%%%
\section{Covers of $\Sp_{2r}$ and $\GL_r$} \label{S:p}

The goal of this section is to show that for odd-fold covers of $\Sp_{2r}$, Conjecture \ref{C:main} agrees with the one studied by Friedberg and Ginzburg, see \cite{FG5, FG6}. For covers of $\GL_r$, we also explicate the equality \eqref{E:main1} in Conjecture \ref{C:main}, which has been verified by Y.-Q. Cai \cite{Cai1} and Savin \cite{Sav2} (unpublished).

In this section, we continue to assume that $\Theta(\pi^\dag, \nu)$ is an unramified theta representation. We follow the notations of \S \ref{SS:Sp}, and consider $\wt{\Sp}_{2r}^{(n)}$ for odd $n$, which is associated with $Q(\alpha_r^\vee)=1$. In this case, $n_\alpha=n$ is constant on $\alpha\in \Delta$. Recall that every nilpotent orbit of $\Sp_r$ is parametrized by a symplectic partition $(c_1^{p_1} c_2^{p_2} ... c_k^{p_k})$ of $2r$ such that $c_1 > c_2 > ... > c_k \gest 1$  and $p_i$ is even if $c_i$ is odd.
For every partition $(c_1^{p_1} c_2^{p_2} ... c_k^{p_k})$ of $2r$, denote by 
$$(c_1^{p_1} c_2^{p_2} ... c_k^{p_k})_\Sp$$
its symplectic collapse. For every $n\in \N$, we have $2r=qn + t$ with $q\in \Z_{\gest 0}$ and $0\lest t < n$. Denote by
$$\mca{O}_{2r, n} = (n^q t)_\Sp,$$
where we omit $t$ if it is zero.

\begin{thm} \label{T:Sp}
Let $\wt{\Sp}_{2r}^{(n)}$ be the cover with $n$ odd. Let $\nu = \rho/n \in X\otimes \R$ be the unique exceptional character. Then
$$\mca{O}_{\rm Spr}(j_{W_\nu}^W(\varepsilon_\nu))=\mca{O}_{2r, n}.$$
Thus, \eqref{E:main1} in Conjecture \ref{C:main} is equivalent to \cite[Conjecture 2.2]{FG5}.
\end{thm}
\begin{proof}
The result follows from a direct computation of the root subsystem $\Phi_\nu$, the Macdonald representation $j_{W_\nu}^W(\varepsilon_\nu)$ and the nilpotent orbit $\mca{O}_{\rm Spr}(j_{W_\nu}^W(\varepsilon_\nu))$. Using the notation in \S \ref{SS:Sp}, we define
$$\Phi_\nu^{\vee} = \set{\beta^\vee \in \Phi^\vee: \ \angb{\nu}{\beta^\vee} \in \Z}$$
and
$$\Phi_{\nu, j}^{+,\vee}= \Phi_{\nu}^\vee \cap \Phi_{+, j}^\vee \text{ for every } j\in \set{\rm I, II, III}.$$
Clearly, the root system $\Phi_\nu$ is dual to that of $\Phi_\nu^\vee$.
We use the standard notations as in \cite[Page 267]{Bou} such that $\alpha_i^\vee=e_i - e_{i+1}$ for $1\lest i < r$ and $\alpha_r^\vee = e_r$.
For every $i\in [1, r]$, we set $i' = r+1 -i$. For every $x\in \R$, denote by $\floor{x} \in \Z$ its integral part. It is easy to see
$$\begin{aligned}[t]
\Phi_{\nu, \rm I}^{+,\vee} & =\set{e_i: \ i'/n \in [1, \floor{r/n}]},\\
\Phi_{\nu, \rm II}^{+,\vee} & =\set{e_i - e_j: \ n|(j-i), 1\lest i < j \lest r}, \\
\Phi_{\nu, \rm III}^{+,\vee} & =\set{e_i + e_j: \ n|(j' + i'), 1\lest i < j \lest r}.\\
\end{aligned}$$
Write
$$r=an + b \text{ with } a\in \Z_{\gest 0} \text{ and } 0\lest b < n.$$
We proceed by considering the following four cases, which exhaust all possibilities:
\begin{enumerate}
\item[(i)] $n>r$,
\item[(ii)] $r=a n$ with $a\in \N$,
\item[(iii)] $n<r, (n+1)/2\lest b < n$,
\item[(iv)] $n<r, 0< b \lest (n-1)/2$.
%\item[(v)] $n<r, b=0$.
\end{enumerate}
If $n\gest 2r+1$, then $\Theta(\pi^\dag, \nu)$ is generic and the assertion holds by Theorem \ref{T:generic}. The case $n=1$ is also trivial. Thus, we assume $1< n< 2r$, and for simplicity of notations, we will also write
$$n= 2m+1 \text{ with } m\in [1, r-1].$$

For case (i), it is clear that $\Phi_{\nu, \rm I}^{+,\vee} = \Phi_{\nu, \rm II}^{+,\vee} = \emptyset$, and 
$$\Phi_{\nu, \rm III}^{+,\vee} = \set{e_i + e_{n-i}: \ i \in [m+1, r]}.$$
The root subsystem $\Phi_\nu$ is then of type
$$\underbrace{C_1 \times C_1  \times ... \times C_1}_\text{$(r-m)$ copies}$$
Thus
$$j_{W_\nu}^W(\varepsilon_\nu) = \sigma(m; r-m),$$
where $(\xi; \eta)$ is an ordered partition of $r$ and parametrizes an irreducible representation $\sigma(\xi; \eta)$ of $W$, and every element in $\Irr(W)$ corresponds to such an ordered partition, see \cite[Page 379]{Car}, \cite{Lus79} or \cite[\S 5.5]{GePf}. By a computation with the Lusztig symbol (see \cite[Page 419]{Car} or \cite[\S 10.1]{CM}), we see that 
$$\mca{O}_{\rm Spr}(j_{W_\nu}^W(\varepsilon_\nu)) = (n-1, 2r+1-n)_\Sp = (n-1, 2r+1-n).$$
On the other hand, we have $\mca{O}_{2r, n} =(n, 2r-n)_\Sp = (n-1, 2r+1-n)$ as well.

For case (ii), where $r=an$ and thus $q=2a$, a direct computation shows that the root subsystem $\Phi_\nu$ is of type
$$\underbrace{A_{2a-1} \times A_{2a-1}  \times ... \times A_{2a-1}}_\text{$m$ copies} \times C_a,$$
which is a subsystem inside
$$ \underbrace{C_{2a} \times C_{2a}  \times ... \times C_{2a}}_\text{$m$ copies} \times C_a \subset C_r.$$
Here $A_{k-1}$ is the usual subsystem inside $C_k$, and one has
\begin{equation} \label{E:A-DC}
j_{A_{k-1}}^{C_k}(\varepsilon) = j_{D_{\floor{(k+1/2)}} \times C_{\floor{k/2}}}^{C_k} (\varepsilon),
\end{equation}
see \cite[Remarks 3]{Mac1}, where the result was stated for type $B$ groups but also holds for type $C$. Here $j_{A_{k-1}}^{C_k}$ means the $j$-induction from the Weyl group of $A_{k-1}$ to the Weyl group of $C_k$; similarly for the right hand side of \eqref{E:A-DC}. Also, we adopt the convention that $D_1 =\emptyset$. Since the $j$-induction is transitive and compatible with direct products, it follows from \eqref{E:A-DC} that $j_{W_\nu}^W(\varepsilon_\nu)$ equals to the Macdonald representation induced from the subsystem
$$\underbrace{D_a \times D_a  \times ... \times D_a}_\text{$m$ copies} \times \underbrace{C_a  \times C_a \times ... \times C_a}_\text{$(m+1)$ copies}$$
of $C_r$. Hence, $j_{W_\nu}^W(\varepsilon_\nu) = \sigma(m^a; (m+1)^a)$. It now follows from a computation with the Lusztig symbol that
$$\mca{O}_{\rm Spr}(j_{W_\nu}^W(\varepsilon_\nu)) = (n^{2a}) = \mca{O}_{2r, n},$$
as desired.

For case (iii), the root subsystem of $\Phi_\nu$ is of type
$$\underbrace{A_{2a+1} \times A_{2a+1}  \times ... \times A_{2a+1}}_\text{$(b-m)$ copies}  \times \underbrace{A_{2a} \times A_{2a}  \times ... \times A_{2a}}_\text{$(2m-b)$ copies} \times C_a.$$
It follows from \eqref{E:A-DC} that $j_{W_\nu}^W(\varepsilon_\nu)$ is equal to the Macdonald representation induced from the root subsystem
$$\underbrace{D_{a+1} \times D_{a+1}  \times ... \times D_{a+1}}_\text{$m$ copies}  \times \underbrace{C_{a+1} \times C_{a+1}  \times ... \times C_{a+1}}_\text{$(b-m)$ copies} \times \underbrace{C_{a} \times C_{a}  \times ... \times C_{a}}_\text{$(n-b)$ copies}.$$
This gives that $j_{W_\nu}^W(\varepsilon_\nu) = \sigma(m^{a+1}; (n-m)^a (b-m))$. A computation with the Lusztig symbol gives that
$$\mca{O}_{\rm Spr}(j_{W_\nu}^W(\varepsilon_\nu)) = (n^{2a} \cdot (n-1) \cdot 2(b-m))_\Sp = (n^{2a} \cdot (n-1) \cdot 2(b-m)).$$
Now we have
$$\mca{O}_{2r, n} = (n^{2a+1} \cdot (2b-n))_\Sp = \mca{O}_{\rm Spr}(j_{W_\nu}^W(\varepsilon_\nu)),$$
as desired.

Lastly, for case (iv), the root subsystem of $\Phi_\nu$ is of type
$$\underbrace{A_{2a} \times A_{2a}  \times ... \times A_{2a}}_\text{$b$ copies}  \times \underbrace{A_{2a-1} \times A_{2a-1}  \times ... \times A_{2a-1}}_\text{$(m-b)$ copies} \times C_a.$$
Again, \eqref{E:A-DC} gives that $j_{W_\nu}^W(\varepsilon_\nu)$ is equal to the Macdonald representation induced from the root subsystem
$$\underbrace{D_{a+1} \times D_{a+1}  \times ... \times D_{a+1}}_\text{$b$ copies}  \times \underbrace{D_{a} \times D_{a}  \times ... \times D_{a}}_\text{$(m-b)$ copies} \times \underbrace{C_{a} \times C_{a}  \times ... \times C_{a}}_\text{$(m+1)$ copies}.$$
This gives that 
$$j_{W_\nu}^W(\varepsilon_\nu) = \sigma(bm^a; (m+1)^a).$$
A computation with the Lusztig symbol gives that
$$\mca{O}_{\rm Spr}(j_{W_\nu}^W(\varepsilon_\nu)) = ((2b) \cdot n^{2a})_\Sp = ((2b) \cdot n^{2a}).$$
On the other hand, since $2r=(2a)n + 2b$ with $2b<n$, we have
$$\mca{O}_{2r, n} =((2b) \cdot n^{2a}) = \mca{O}_{\rm Spr}(j_{W_\nu}^W(\varepsilon_\nu))$$
as well. This completes the proof.
\end{proof}

We note that \cite[Conjecture 2.2]{FG5} (and thus also Conjecture \ref{C:main} by the above equivalence) has been verified in some cases, see \cite[Theorem 1]{FG6}.

\begin{rmk}
Consider the persistent cover $\wt{\Sp}_{2r}^{(4)}$ associated with $Q(\alpha_r^\vee) = 1$ (see Lemma \ref{L:4n-pers}). Let $\nu\in X\otimes \R$ be the unique exceptional character. One has $\tilde{Y}_{Q,4} = 2\cdot Y$ and the saturation of $\nu$ is $\tnu = \rho/2$. A similar computation as in Theorem \ref{T:Sp} shows that the root subsystem $\Phi_\tnu$ is $D_{a} \times C_{a}$ if $r=2a$ is even, and is equal to $D_{a+1} \times C_a$ if $r=2a+1$ is odd. Thus we have
$$\mca{O}_{\rm Spr}(j_{W_\tnu}^W(\varepsilon_\tnu)) = (2^{r}),$$
and \eqref{E:main1} in Conjecture \ref{C:main} in this case was proved by Leslie in \cite[Theorem 1.3]{Les}.
\end{rmk}

Consider the Kazhdan-Patterson cover $\wt{\GL}_r$ associated to a pair $ (\bfp, \bfq) \in \Z \times \Z$ with $2\bfp - \bfq=-1$. More precisely, we have the quadratic form $Q$ on $Y$ determined by  $Q(e_i) = \bfp$ and $B_Q(e_i, e_j) = \bfq, i\ne j$; this gives rise to the $n$-fold cover $\wt{\GL}_r$. Here $Q(\alpha^\vee) = -1$ for every root $\alpha$. Every $\wt{\GL}_r$ is saturated and thus persistent. 

\begin{thm}[{\cite[Theorem 1.2]{Cai1}}] \label{T:GL}
For a Kazhdan-Patterson $n$-fold cover $\wt{\GL}_r$, one has
\begin{equation} \label{E:GL}
\mca{O}_{\rm Spr}(j_{W_\nu}^{W} (\varepsilon_\nu)) =(n^q t),
\end{equation}
where $r= qn + t$ with $0\lest t < n$; moreover, \eqref{E:main1} in Conjecture \ref{C:main} holds for such cover.
\end{thm}
\begin{proof}
It suffices to prove the equality \eqref{E:GL}, which coupled with \cite[Theorem 1.2]{Cai1} verifies Conjecture \ref{C:main}.
If $\nu \in X\otimes \R$ is exceptional, then we have 
$$\phi(\nu) = \rho/n \in X^{sc}\otimes \R,$$
and $W_\nu$ is just the integral Weyl subgroup of $\phi(\nu)$ with respect to $X^{sc}$. A similar (and in fact simpler) computation as in Theorem \ref{T:Sp} shows that the root subsystem $\Phi_\nu$ is of the form
$$\underbrace{A_{q} \times A_{q}  \times ... \times A_{q}}_\text{$t$ copies}  \times \underbrace{A_{q-1} \times A_{q-1}  \times ... \times A_{q-1}}_\text{$(n-t)$ copies}.$$
Thus the Macdonald representation $j_{W_\nu}^W (\varepsilon_\nu)$ is associated with the partition $(n^qt)$ of $r$ (see \cite[\S 11.4]{Car}), which parametrizes exactly the orbit $\mca{O}_{\rm Spr}(j_{W_\nu}^W (\varepsilon_\nu))$ by the Springer correspondence. This completes the proof.
\end{proof}

\begin{rmk}
In Theorems \ref{T:Sp} and \ref{T:GL}. the assumption that $Q(\alpha_r^\vee) =-1$ and $Q(\alpha^\vee)=-1$) for $\Sp_{2r}$ and $\GL_r$ respectively is not essential. Indeed, the argument and result hold for general quadratic form as well: for $\wt{\Sp}_{2r}$, one only needs to assume that $n_{\alpha_r}$ is odd; for $\GL_r$, it could be any quadratic form and thus works for arbitrary Brylinski--Deligne covers of $\GL_r$.
\end{rmk}

As a last example, we consider covers of $\SO_{2r+1}$. Let
$$\SO_{2r+1} \into \SL_{2r+1}$$
be the natural embedding. Consider the $n$-fold cover $\wt{\SL}_{2r+1}^{(n)}$ associated with $Q(\alpha^\vee)=1$ for any coroot $\alpha$. By restriction, one obtains a cover $\wt{\SO}_{2r+1}^{(n)}$. For $n=2$, the double cover $\wt{\SO}_{2r+1}^{(2)}$ is not a linear group but has the special property that its covering torus is abelian, see \cite[Example 2.10]{GSS1}.

The four-fold cover $\wt{\SO}_{2r+1}^{(4)}$ and its associated theory of theta liftings were investigated in the work of Bump, Friedberg and Ginzburg \cite{BFrG2, BFrG}. We have
$$Y_{Q,4} = 2Y = Y_{Q,4}^{sc}$$
and thus the dual group of $\wt{\SO}_{2r+1}^{(4)}$ is $\SO_{2r+1}$. In particular, $\wt{\SO}_{2r+1}^{(4)}$ is saturated and thus persistent. Using notations in \S \ref{SS:Br}, it is easy to see that
$$\nu=\omega_r + \sum_{i=1}^{r-1} (\omega_i/2) \in X\otimes \R$$
is the unique exceptional character. A simple computation shows that
$$
\Phi_\nu \text{ is of type }
\begin{cases}
B_m \times B_m  & \text{ if } r=2m, \\
B_{m+1} \times B_m & \text{ if } r=2m+1.
\end{cases}
$$
It follows that 
$$\mca{O}_{\rm Spr}(j_{W_\nu}^W(\varepsilon_\nu))  =
\begin{cases}
 (2^{2m} 1) & \text{ if } r=2m, \\
(2^{2m} 1^3) & \text{ if } r=2m+1.
\end{cases}
$$
In this case, \eqref{E:main1} in Conjecture \ref{C:main} was proved in \cite[Theorem 4.2]{BFrG2}. In particular, $\Theta(\pi^\dag, \nu)$ is a minimal representation for $r=2, 3$. We also remark that for the double cover $\wt{\rm GSpin}_{2r+1}^{(2)}$, the equality \eqref{E:main1} in Conjecture \ref{C:main} was proved by Kaplan \cite[Theorem 1]{Kap004}. The computation is similar to the above, and we omit the details.

%%%
\section{The archimedean analogue} \label{S:arch} 
In this section, we discuss about an analogue of Conjecture \ref{C:main} for archimedean $F$, which in fact motivated us to consider the $p$-adic case in the previous sections.

As a first example, consider $F=\C$ and the three-fold cover $\wt{G}_2^{(3)}$ of $G_2$, which splits over $G_2$. It follows from \cite{Sav4} that $\Theta(\pi^\dag, \nu)$ is a minimal representation. Now we check that the formal analogue of Conjecture \ref{C:main} holds in this case. A simple computation gives 
$$\Phi_\nu^+ = \set{\alpha_1^\vee, 3\alpha_2^\vee + \alpha_1^\vee, 3\alpha_2^\vee + 2\alpha_1^\vee},$$
the set of long positive coroots in $\Phi_+^\vee$, and furthermore we have
$$j_{W_\nu}^W(\varepsilon_\nu) = \phi_{1, 3}'',$$
where $\phi_{1, 3}''$ is the standard-labelled representation with character values given in \cite[Page 412]{Car}. This gives
$$\mca{O}_{\rm Spr}(\phi_{1, 3}'') = \mca{O}_{\rm min},$$
see \cite[Page 427]{Car}. Thus, we have a formal analogue of Conjecture \ref{C:main} for this $\Theta(\pi^\dag, \nu)$.

In the remaining of this section, we assume $F=\C$ or  $\R$ (and in the latter case $n=2$) and summarize some known results. In particular, the main result is Theorem \ref{p:diagram}. It gives the desired archimedean analogue of Conjecture \ref{C:main}, see the discussion in \S \ref{SSS:R-spl}.

\subsection{Wave front set and associated variety} \label{s:WFAV}
{Let $G$ be a reductive Lie group with maximal compact subgroup $K$, and let $G_\C$ and $K_\C$ be the complexification of $G$ and $K$. 
Denote $\mfr{g}$, $\mfr{k}$,  $\mfr{g}_\C$ and $\mfr{k}_\C$ to be the Lie algebra of $G, K, G_\C$ and $K_\C$, respectively.  Let $\pi$ be an irreducible admissible representation $\pi$ of $G$, we briefly recall the two invariants of cycles associated with it, one defined analytically and the other algebraically.

In \cite{BV3}, one has for $\pi$ an asymptotic expansion for the 
character expansion $\chi_\pi$ in a neighborhood of 0 in $\mfr g$, of the form 
$$
\chi_\pi \sim \sum _{i=-r} ^\infty D_i
$$
with $\set{D_i}$ being a set of tempered distributions on $\mfr g$. The asymptotic support 
$${\rm AS}(\chi_\pi)\subset \mfr g^*$$
 is defined to be the union of the supports of the 
Fourier transforms $\widehat D_i$. It is known that 
${\rm AS}(\chi_\pi)$ is a union of nilpotent orbits. Identifying $\mfr{g}$ with its dual $\mfr{g}^*$ by the Cartan-Killing form, we view ${\rm AS}(\chi_\pi) \subset \mfr{g}$ and define 
$$
\mca{N}_{\rm tr} (\pi) = \set{ \mca O\in \mca N: \mca O\subset {\rm AS}(\chi_\pi) },
$$
where $\mca N$ denotes the set of nilpotent orbits in $\mfr g$. 
The set $\mca{N}_{\rm tr}(\pi)$ coincides with the wave front set of $\pi$ defined by Howe in \cite{How2},  as  was
proved by Rossmann (see \cite{Ros1, Ros2}).

The Fourier transform of the leading term in the asymptotic expansion of $\chi_\pi$ is a linear combination of invariant measures $\mu_\mca{O}$ taking the form
 $$\sum_{\mca{O} \in \mca{N}_{\rm tr}^{\rm max}(\pi)} c_\mca{O} \cdot \mu_\mca{O},$$
 where $0\ne c_\mca{O} \in \C$ and $\mca{N}_{\rm tr}^{\rm max}(\pi) \subset \mca{N}_{\rm tr}(\pi)$ is a certain subset. One thus defines the wavefront cycle of $\pi$ as the following finite sum
 $$
 {\rm WFC}(\pi) = \sum_{\mca{O} \in \mca{N}_{\rm tr}^{\rm max}(\pi)} c_\mca{O}\cdot \mca{O}.
 $$ 

On the other hand, we can attach nilpotent orbits to $\pi$ with an algebraic method (see \cite{Vog5} for more details). Consider the Harish-Chandra $(\mfr{g}_\C, K)$-module $V$ associated with $\pi$. Let $U(\mfr{g}_\C)$ be the enveloping algebra of $\mfr{g}_\C$, and let  $U_k\subset U(\mfr{g}_\C), k\gest 0$ be the subspace spanned by products of at most $k$ elements of $\mfr{g}_\C$.   Then $V$ admits a ``good filtration" 
$$V_0\subset V_1\subset \cdots V_k\subset\cdots,$$
with $V_k=U_k\cdot V_0$. Relative to this filtration we form an associated graded space ${\rm gr}(V)$ as a module over ${\rm gr}(U(\mfr{g}_\C))$. We identify 
${\rm gr}(U(\mfr{g}_\C))$ with  the symmetric algebra $S(\mfr{g}_\C)$ by the Poincar\'{e}--Birkhoff--Witt theorem.
Moreover, ${\rm gr}(V)$ can be viewed as a module over $S(\mfr{g}_\C/ \mfr{k}_\C)$ since the action of $\mfr{k}_\C$ preserves the filtration of $V$ and hence the ideal generated by $\mfr{k}_\C$ in $S(\mfr{g}_\C)$ annihilates 
${\rm gr}(V)$. 
The associated variety of $\pi$ is defined to be 
$$
{\rm AV}(\pi) = \{ \lambda \in \mfr{g}_\C ^*: p(\lambda) =0 \text{ whenever } p\in\text{Ann}({\rm gr}(V))  \},
$$
which is a subset of $(\mfr{g}_\C/\mfr{k}_\C)^*$.  By Kostant and Rallis \cite{KR}, ${\rm AV}(\pi)$ is a (finite) union of nilpotent orbits in $(\mfr{g}_\C / \mfr{k}_\C)^*$. 
Identifying coadjoint orbits with adjoint orbits, and using the Sekiguchi correspondence to identify the nilpotent $K_\C$-orbits in $(\mfr{g}_\C /\mfr{k}_\C)^*$ with the nilpotent $G$-orbits in $\mfr{g}^*$, we 
define the set 
$$
\mca N _{\rm alg}(\pi) = \set{ \mca O\in \mca N\mid \mca O \subset {\rm AV}(\pi)}.
$$
Let $\mca N_{\rm alg} ^{\rm max}(\pi)$ be the set of maximal elements in  $\mca N_{\rm alg}(\pi)$. Then one can define the associated cycle of $\pi$
to be $$
{\rm Ass}(\pi) = \sum _{\mca O\in \mca N_{\rm alg} ^{\rm max}(\pi)} b_{\mca O} \cdot \mca O,
$$ 
where $b_{\mca O}$ denotes the rank of the sheaf ${\rm gr}(V)$ along $\wt{\mca O^{K_\C}}$.
It follows from the result of Schmid and Vilonen \cite{SV} that 
$${\rm WFC}(\pi)={\rm Ass}(\pi);$$
in particular, $\mca{N}_{\rm tr} ^{\rm max}(\pi) = \mca N_{\rm alg} ^{\rm max}(\pi)$.

Motivated from this, we will concentrate on the algebraic invariants $\mca N_{\rm alg}$, $\mca N_{\rm alg} ^{\rm max}$ for 
the following reasons:
\begin{enumerate}
\item[--] If we consider a cover $\wt{G}$ of $G$, then the Barbasch--Vogan character expansion of an irreducible genuine $\pi$ is expected to hold. However, as the details have not been checked in the literature, we exert some caution and do not make an assumption of it.
\item[--] On the other hand, the invariants defined algebraically could be well adapted for covering groups, as the algebraic invariants depend more on the action of the Lie algebra and enveloping algebra of $\wt{G}$, which are the same as $G$.
\end{enumerate}
Thus, in the remaining of this section, we will explore extensively the pertinent work obtained from the algebraic method.

\subsection{Some invariants} 
For more details of this subsection, we refer the reader to \cite{BV4, BV5, BV6}.  We assume that $\mfr{g}_\C$ is an arbitrary complex semisimple Lie algebra, with enveloping algebra $U(\mfr g_\C)$.  Let $\mfr{h}$ be a Cartan subalgebra of $\mfr{g}_\C$, and $\Phi(\mfr{g}_\C, \mfr{h})$ the associated roots. Let $W$ be the Weyl group.

\subsubsection{Primitive ideals} \label{s:prim}
An ideal $I$ of $U(\mfr{g}_\C)$ is called a primitive ideal if it is the annihilator of a simple $U(\mfr{g}_\C)$-module V. We say that the primitive ideal $I={\rm Ann}(V)$ has infinitesimal character $\lambda\in \mfr{h}^*$ if 
$V$ has infinitesimal character $\lambda$.  Let 
${\rm Prim}(U(\mfr{g}_\C))$ 
be the set of primitive ideals in $U(\mfr{g}_\C)$, and let
$${\rm Prim}_\lambda(U(\mfr{g}_\C))$$
be the subset of  those with infinitesimal character $\lambda$.
Define 
\begin{align*}
\Phi_\lambda &=\set{\alpha\in \Phi(\mfr{g}_\C,\mfr{h}): \ \angb{\alpha^\vee}{\lambda} \in \Z }, \\
\Phi_\lambda^+ &=\set{\alpha\in \Phi_\lambda: \angb{\alpha^\vee}{\lambda}  >0 }, \\
\Delta_\lambda & = \text{ simple roots of } \Phi_\lambda ^+ \\
W_\lambda &=W(\Phi_\lambda), \text{ the integral Weyl group of  } \lambda.
\end{align*} 
We choose a positive root system $\Phi^+ =\Phi^+(\mfr{g}_\C , \mfr{h})$ such that
$$\Phi^+ \supseteq -\Phi_\lambda ^+;$$
that is, we assume $\lambda$ is negative. 
For $w\in W_\lambda$, we put
\begin{itemize}
\item[(i)] $\mfr b = \mfr h +\mfr n$, with $\Phi(\mfr n, \mfr h)=\Phi^+$,
\item[(ii)] $M(w\lambda) = U(\mfr{g}_\C)\otimes_{U(\mfr{b})} \C_{w\lambda -\rho}$,
\item[(iii)] $L(w\lambda) =$ the irreducible quotient of $M(w\lambda)$,
\item[(iv)] $I(w\lambda) ={\rm Ann}(L(w\lambda))$, the annihilator of $L(w\lambda)$ in $U(\mfr{g}_\C).$
\end{itemize}
We may write $I(w):=I(w\lambda)$ for brevity, whenever $\lambda$ is understood.

\begin{prop}[\cite{Duf}] \label{p:prim} 
The map $W_\lambda \to {\rm Prim}_\lambda(U(\mfr{g}_\C))$ given by $w\mapsto I(w)$ is surjective. 
\end{prop}

The following result associates a primitive ideal to a special Weyl group representation. 
\begin{thm}[{\cite[Theorem D]{BV4}, \cite[Theorem 1.1]{BV5}}] \label{t:Goldie}
Suppose $\lambda\in\mfr{h}^*$ is regular, and $I\in {\rm Prim}_\lambda(U(\mfr{g}_\C))$. Then Joseph's Goldie rank representation \cite{JosII} $\sigma _I\in \Irr(W_\lambda)$ is a 
special representation of $W_\lambda$ in the sense of Lusztig \cite{Lus79}; also, every special representation of $W_\lambda$ occurs as a Joseph's Goldie rank representation.
\end{thm}
We define the associated variety of $I\in {\rm Prim}_\lambda(U(\mfr{g}_\C))$  as in \S \ref{s:WFAV}. More precisely, let ${\rm gr}(I) \subset S(\mfr{g}_\C)$ be the associated graded ideal in the symmetric algebra $S(\mfr{g}_\C)$. Then ${\rm AV}(I)$, the associated variety of $I$, is defined to be the zero variety of ${\rm gr}(I)$ inside $\mfr{g}_\C^*$. It is well-known that ${\rm AV}(I)$ is the closure of a single nilpotent orbit in $\mfr{g}_\C$, which we recall is identified with $\mfr{g}_\C^*$.

\begin{rmk}
Consider an irreducible Harish-Chandra $(\mfr{g}_\C, K)$-module $(\pi, V)$, where $K$ is the maximal compact subgroup of a certain group $G$ with ${\rm Lie}(G)\otimes\C=\mfr{g}_\C$.  Here $\pi$ can be viewed as an irreducible admissible representation of $G$. 
Let $I_\pi :={\rm Ann}_{U(\mfr{g}_\C)}(V)$. Then 
$${\rm AV}(I_\pi)= G_\C \cdot {\rm AV}(\pi). $$
This remark works for both linear and covering groups.
\end{rmk}  

\subsubsection{Cells and the Springer correspondence} \label{s:cell}
Assume $\lambda\in\mfr h^*$ is regular. We recall the cell decomposition on $W_\lambda$ as follows. We first define relations $\lest_{\rm L}$ and $\lest_{\rm R}$ on elements in $W_\lambda$ by setting 
$$
w_1\lest_{\rm L} w_2 \iff I(w_1) \subseteq I(w_2), \quad \text{(see notation in Proposition \ref{p:prim})}
$$
and 
$$
w_1\lest_{\rm R} w_2 \iff  w_1 ^{-1}\lest_{\rm L} w_2 ^{-1}. 
$$
The smallest relation containing $\lest_{\rm L}$ and $\lest_{\rm R}$ is denoted by $\lest_{\rm LR}$.  The set
$$
\wt{\mca{C}}^{\rm L}_w = \set{w' \in W_\lambda: w \lest_{\rm L} w'}
$$
is called the left cone over $w$. Similarly, we define $\wt{\mca{C}}^{\rm R}_w$, $\wt{\mca{C}}^{\rm LR}_w$ using $\lest_{\rm R}$ and $\lest_{\rm LR}$ respectively.

We also define
$$
w_1 \approx_{\rm L} w_2 \iff I(w_1)=I(w_2) \iff w_1\lest_{\rm L} w_2 \lest_{\rm L} w_1,
$$
and similarly $\approx_{\rm R}, \approx_{\rm LR}$. The equivalence classes of $\approx_{\rm L}$ (respectively, $\approx_{\rm R}$ and $\approx_{\rm LR}$) are called the left (respectively, right and two-sided) cells.  The three cells containing $w$ are denoted by 
$$\mca{C}_w^{\rm L}, \mca{C}_w^{\rm R} \text{ and } \mca{C}_w^{\rm LR},$$
respectively. 

To define the cell decomposition on $\Irr(W_\lambda)$, we first consider
\begin{align} 
\wt{V}^{\rm L}_w &:= \bigoplus _{w ' \in \wt{\mca{C}}^{\rm L}_w} \C \cdot L(w' \lambda) \subseteq \C[W_\lambda], \label{e:left1}\\
K^{\rm L}_w & :=  \bigoplus_{ \substack{ w' \in \wt{\mca{C}}^{\rm L}_w  \\ w' \not\in  \mca{C}^{\rm L}_w  }} \wt{V}^{\rm L}_{w'} \subseteq \wt V^{\rm L}_w, \label{e:left2} \\
V^{\rm L}_w &:= \wt V^{\rm L}_w /K^{\rm L}_w. \label{e:left3}
\end{align}
Similarly, we define the analogues of these objects decorated by R or LR for their superscripts. 
Here $V^{\rm L}_w$ affords a natural representation of $W_\lambda$, similarly for $V_w^{\rm R}$ and $V_w^{\rm LR}$, see \cite[Corollary 2.11]{BV5}. 

If $I(w)\in {\rm Prim}_\lambda(U(\mfr{g}_\C))$, then $V_w^{\rm L}$ (resp. $V_w^{\rm LR}$) is called the left cell (resp. double cell) representation of $W_\lambda$ associated with $I(w)$. For  $\sigma_1,\sigma_2 \in \Irr(W_\lambda)$, we use
$$\sigma_1\lest_{\rm LR} \sigma_2$$ to mean that
$\sigma_1\otimes \sigma_1$ (the double representation) occurs in $V_{w}^{\rm LR}$ and that $\sigma_2\otimes \sigma_2$ occurs in $\wt{V}^{\rm LR}_w$. Thus,  
$$
\sigma_1\approx_{\rm LR} \sigma_2 \iff \sigma_1\otimes \sigma_1  \text{ and }  \sigma_2\otimes \sigma_2 \text{ occur in a common }  V_w^{\rm LR}. 
$$
The double cells in $\Irr(W_\lambda)$ are the $\approx_{\rm LR}$ equivalence classes.  

Consider the multiset 
$$\set{ \dim(\sigma)\cdot \sigma \mid \sigma \in \Irr(W_\lambda)}.$$
 A PI cell in this multiset is a submultiset $\set{ m_\sigma\cdot \sigma }$ such that $\sum m_\sigma \sigma$
is a left cell representation. Denote by $\mca{P}(\mca{C}^{\rm L})$ the PI cell attached to a left cell $\mca{C}^{\rm L}$. Note that $\mca{P}(\mca{C}^{\rm L})\simeq V^{\rm L}_w$ for some $w$. Clearly, we have a decomposition of $\Irr(W_\lambda)$ into the double cells and
$$
\set{\dim(\sigma) \cdot \sigma:  \sigma \in \Irr(W_\lambda)} =  \bigsqcup_{\text{left cells } \mca{C}^{\rm L}} \mca{P}(\mca{C}^{\rm L}).
$$

\begin{prop}  \label{p:cell} \cite[Corollary 2.16]{BV5}
\begin{itemize}
\item[(a)] Each double cell  in $\Irr(W_\lambda)$ contains exactly one special representation (equivalently, a Goldie rank representation) of $W_\lambda$. 
\item[(b)] Each {\rm PI} cell (or every left cell representation) contains exactly one Goldie rank representation with multiplicity one. 
\end{itemize}
\end{prop}

For $w\in W_\lambda$, the $\tau$-invariant of $w$ is defined to be
$$
\tau (w) = \set{\alpha\in \Phi_\lambda ^+ \mid w\alpha\not\in \Phi_\lambda ^+ }\cap \Delta _\lambda.
$$
The Borho--Jantzen--Duflo $\tau$-invariant of $I(w\lambda)\in {\rm Prim}_\lambda(U(\mfr{g}_\C))$ is then
 $$\tau (I(w\lambda)):=\tau (w).$$
 It is known (see \cite[Corollary 2.19]{BV5}) that this
$\tau$-invariant depends only on the ideal $I(w\lambda)$, and thus we have a well-defined order-preserving map 
$$\tau: {\rm Prim}_\lambda(U(\mfr{g}_\C)) \longrightarrow \Delta_\lambda.$$

\begin{prop} \label{p:tau}
Suppose $I:=I(w\lambda)\in {\rm Prim}_\lambda (U(\mfr{g}_\C))$. If the $\tau$-invariant of $I$ is maximal (i.e., $\tau(I) =\Delta_\lambda$), then $\sigma_I = \varepsilon_{W_\lambda}$.
\end{prop}
\begin{proof}
Let $I={\rm Ann}(L(w\lambda))$. It follows from \cite[Proposition 2.20]{BV5} that
$$w_\alpha\cdot L(w\lambda)=-L(w\lambda)$$
 for all $\alpha\in \tau (w)$, where $w_\alpha$ is the simple reflection of $\alpha$.
This means that the left cell representation $V_w^{\rm L}$ contains a sign representation of $W_\lambda$. Since $\varepsilon_{W_\lambda}$ is special, we must have $\sigma_I =\varepsilon_{W_\lambda}$ by  Proposition \ref{p:cell}.
\end{proof}

Now we relate these notions to nilpotent orbits. Let $\mca{N}_{\mfr{g}_\C}$ be the set of nilpotent classes of $\mfr{g}_\C$. Let $\mca{O}\in\mca{N}_{\mfr{g}_\C}$. Recall that the Springer correspondence gives an injective map 
$$ \begin{tikzcd}
{\rm Spr}_\mbm{1}^{-1}: \mca{N}_{\mfr{g}_\C} \ar[r, hook] & \Irr(W).
\end{tikzcd}$$
Every special representation $\sigma$ of $W$ (in the sense of Lusztig) lies in the image of ${\rm Spr}_\mbm{1}^{-1}$, and thus is associated with the nilpotent orbit $\mca{O}_{\rm Spr} (\sigma)$ via the Springer correspondence.

If $\lambda \in \mfr{h}^*$ is integral, then $W_\lambda = W$. By Proposition \ref{p:cell}, a left cell representation $V^{\rm L}_w$ contains a unique special representation $\sigma (w)\in \Irr(W)$ with 
multiplicity one. We write 
$$\mca O(w):=\mca O_{\rm Spr} (\sigma (w)) \in \mca{N}_{\mfr{g}_\C}$$
for the nilpotent orbit in determined by this left cell representation via the Springer correspondence. 

For general $\sigma\in \Irr(W_\lambda)$, we recall the representation $j_{W_\lambda}^W (\sigma) \in \Irr(W)$ obtained from the truncated $j$-induction (see \cite{LuSp1} or \cite[\S 11.2]{Car}). This $j$-induction takes special representations of $W_\lambda$ to special representations of $W$, see \cite[Proposition 11.3.11]{Car}. The following is the key fact which gives an archimedean analogue to Conjecture \ref{C:main}.

\begin{thm} \label{p:diagram}
Let $\mfr{g}_\C$ be a complex semisimple Lie algebra. Let $I\in {\rm Prim}_\lambda (U(\mfr{g}_\C))$ with $\lambda$ regular. Let $\sigma_I$ be the  Goldie rank representation associated with $I$. Then there
is a nilpotent orbit $\mca{O} \in \mca{N}_{\mfr{g}_\C}$ such that 
$${\rm Spr}_\mbm{1}^{-1}(\mca{O}) = j_{W_\lambda}^W (\sigma_I).$$
Furthermore, the orbit $\mca O$ is dense in ${\rm AV}(I)$. Consequently, we have the following commutative diagram: 
$$\begin{tikzcd}
\mca{O} \ar[r, "{{\rm Spr}_\mbm{1}^{-1}}"] & j_{W_\lambda}^W(\sigma_I) \\
I \ar[r, "\sigma"] \ar[u, "{{\rm AV}}"] & \sigma_I \ar[u, "j"] .
\end{tikzcd} $$
The left vertical arrow in the diagram means ${\rm AV}(I) = \overline{\mca{O}}$.
\end{thm}
\begin{proof}
Let $\mfr m$ be the Levi factor of $\mfr{g}_\C$ determined by $\lambda$. Writing $W(\mfr m)$ for the Weyl group of $\mfr h$ in $\mfr m$, one has $W(\mfr m)=W_\lambda$. It follows from Proposition \ref{p:prim} that a primitive ideal $I\in {\rm Prim}_\lambda(U(\mfr{g}_\C))$ is of the form $$I=I(w\lambda)={\rm Ann}(L(w\lambda))$$
for some $w\in W_\lambda$. Let $\sigma_I\in \Irr(W_\lambda)$ be the Goldie rank representation of $I$ (see Theorem \ref{t:Goldie}). 

Let $V_\mfr{m}^{\rm L}(w)$ be the left cell representation for $w\in W(\mfr m)$. Let $\mca O_{\mfr m} (w)=\mca{O}_{\rm Spr}(\sigma_I)$ be the nilpotent orbit associated with $V_{\mfr m}^{\rm L}(w)$ in $\mca{N}_\mfr{m}$. Write $\mca{O}_\mfr{m}:= \mca{O}_\mfr{m}(w)$ for brevity. One has ${\rm Spr}_\mbm{1}^{-1}(\mca{O}_{\mfr m})=\sigma _I$. 

Since $\sigma _I$ is a  special representation of $W_\lambda$, $j_{W_\lambda}^W (\sigma_I)$ is a special representation of $W$ (by \cite[Proposition 11.3.11]{Car}). Hence, there is a nilpotent orbit 
$\mca{O} \in\mca{N}_{\mfr{g}_\C}$ such that ${\rm Spr}_\mbm{1}^{-1}(\mca{O})=j_{W_\lambda}^W (\sigma_I)$. 
  The orbit $\mca{O}$ is in fact induced from $\mca O_{\mfr m}$ in the sense of Lusztig and Spaltenstein (cf. Definition 4.12 and Proposition 4.14 in \cite{BV6}), that is,
$$
\mca O={\rm Ind}_\mfr{m}^{\mfr{g}_\C} (\mca{O}_\mfr{m}).
$$
Thus, $d(\mca O)=d(\mca O_{\mfr m})$, where 
$d(\mca O)=|\Delta ^+|-\frac{1}{2}\dim \mca O$ and $d(\mca O_{\mfr m})=|\Delta ^+| -\frac{1}{2}\dim \mca O_{\mfr m}$. 
We deduce that $\dim \mca O =\dim \mca O_{\mfr m}$, and therefore $\mca O$ is dense in ${\rm AV}(I)$. 
\end{proof}

\begin{rmk}
Suppose $G$ is a semisimple real (linear or covering) group with complexified Lie algebra $\mfr{g}_\C$. Let $\pi$ be an irreducible admissible representation of $G$ with infinitesimal character $\lambda$. 
By Theorem \ref{p:diagram}, a nilpotent orbit $\mca{O} \in \mca{N}_{\mfr{g}_\C}$ is associated to $\pi$ using the composite of the following maps:
\begin{equation} \label{e:composite}
\pi \mapsto I_\pi:= {\rm Ann}(\pi) \mapsto \sigma_I  \mapsto j_{W_\lambda}^W(\sigma_I) \mapsto \mca O_{\rm Spr} ( j_{W_\lambda}^W(\sigma _I)).
\end{equation}
This composite of functions gives ${\rm AV}(I_\pi)=\wt{\mca{O}_{\rm Spr} ( j_{W_\lambda}^W(\sigma _I))}$.

If $G$ is a complex group viewed as a real group, then $\mfr{g}_\C \simeq \mfr g\times \mfr g$, with ${\rm Lie}(G)=\mfr{g}$. There are two primitive ideals of $U(\mfr{g}_\C)$ associated to an 
irreducible admissible $\pi$, denoted by ${\rm LAnn}(\pi)$ and ${\rm RAnn}(\pi)$, the left annihilator and right annihilator of $\pi$ (see  \S \ref{s:cx} below). We use $I_\pi = {\rm LAnn}(\pi)$ in \eqref{e:composite} to obtain
$$
{\rm AV}(\pi) = {\rm AV}(I_\pi)=\wt{\mca{O}_{\rm Spr}( j_{W_\lambda}^W(\sigma_I))}.
$$
\end{rmk}

In the remaining of this section, we will give some elaborations on the complex and real cases separately.

\subsection{Complex case} \label{s:cx}
Let $G$ be a connected complex semisimple group (viewed as a real group) with Lie algebra $\mfr g$. We first recall the Langlands classification for $G$.  Consider the following data:
\begin{itemize}
\item $\theta$ the Cartan involution, $K=G^\theta$, $\mfr g = \mfr k+\mfr p$ the Cartan decomposition,
\item $\mfr b = \mfr h +\mfr n $ a Borel subalgebra,
\item $\mfr h =\mfr t+ \mfr a $ a Cartan subalgebra, with $\mfr t \subset \mfr{k}$, $\theta |_{\mfr a } = -{\rm id}$,
\item $W =W(\Phi (\mfr g, \mfr h))$ the Weyl group. 
\end{itemize}
For $\lambda_{\rm L}, \lambda_{\rm R}\in \mfr h^*$, we write 
$$
X (\lambda_{\rm L}, \lambda_{\rm R}) = \text{Ind}_B^G (\C_\mu \otimes \C_\nu) 
$$
for the principal series representation with infinitesimal character $(\lambda_{\rm L},\lambda_{\rm R})$, where $\C_\mu\otimes \C_\nu$ is a character of $H$ with 
\begin{align*}
\C_\mu\otimes \C_\nu | _T&= \C _\mu = \C_{\lambda_{\rm L}-\lambda_{\rm R}},\\
\C_\mu\otimes \C_\nu | _A&= \C _\nu = \C_{\lambda_{\rm L}+\lambda_{\rm R}}.
\end{align*}
Set $\wt{X}(\lambda_{\rm L},\lambda_{\rm R})$ to be the unique irreducible subquotient containing $K$ representation of extremal weight $\mu=\lambda_{\rm L}-\lambda_{\rm R}$. Then 
every irreducible admissible representation of $G$ is of the form $\wt X (\lambda_{\rm L},\lambda_{\rm R})$ for some $\lambda_{\rm L}, \lambda_{\rm R}$. 

Note that $\mfr g_\C \simeq \mfr g\times \mfr g$.  For an irreducible admissible representation $\pi=\wt X (\lambda_{\rm L}, \lambda_{\rm R})$,  the annihilator of $\pi$ in $U(\mfr g_\C)$ is of the form 
$$
\text{Ann} (\pi) = I_1\otimes U(\mfr g) +U(\mfr g)\otimes I_2,
$$
where $I_1\in\text{Prim}_{\lambda_{\rm L}} (U(\mfr g))$ and $I_2\in\text{Prim}_{\lambda_{\rm R}} (U(\mfr g))$. These primitive ideals are denoted by
$$
{\rm LAnn}(\pi):=I_1, \ {\rm RAnn}(\pi):=I_2,
$$
called the left and right annihilators of $\pi$.

Suppose $\lambda, -\xi \in \mfr h^*$ are dominant integral. Let $\msc{R}(\lambda,\xi)$ be the Grothendieck group of formal characters of $G$ having infinitesimal character $(\lambda,\xi)$. Then 
$\{ X(\lambda, w\xi)\}$ (or $\{\wt X (\lambda, w\xi)\}$) can be chosen to be a basis for $\msc R (\lambda ,\xi)$. Therefore, $\msc R (\lambda,\xi)$ can be identified with $\C[W]$ as follows:
$$
\sum_{w\in W} c_w w \longleftrightarrow \sum_{w\in W} c_w X(\lambda, w\xi).
$$
The regular representation of $W\times W$ on $\C[W]$ is identified with the coherent continuation of $W\times W$ on $\msc R (\lambda,\xi)$:
\begin{equation} \label{e:reg-action}
(w_1,w_2) \cdot \left (\sum _w c_w w\right ) = \sum_w c_w (w_1ww_2^{-1}) = \sum_w c_{w_1^{-1} ww_2}w. 
\end{equation}
We want to describe a decomposition of $\msc R (\lambda, \xi)$ into cells described as in \S \ref{s:cell}. 
The main fact regarding the left cones is that they are invariant under the action in \eqref{e:reg-action}, i.e.,
$$
(w_1,1)\cdot   \wt X (\lambda,w\xi) = \sum_{w' \in \wt{\mca C}^L _w } a_{w'} \wt X (\lambda, w' \xi). 
$$
Similar to \eqref{e:left1},  \eqref{e:left2} and \eqref{e:left3}, we set 
\begin{align*}
\wt V^{\rm L}_w & = \text{Span}\{ \wt X (\lambda , w' \xi ) \mid w' \in \wt{\mca C}^{\rm L}_w \} \subset \msc R(\lambda, \xi),\\ 
 K^{\rm L}_w & = \text{Span}\{ \wt X (\lambda , w'  \xi )  \mid w' \in \wt{\mca C}^{\rm L} _w,  w' \not\in \mca C ^{\rm L} _w \},  \\
 V^{\rm L}_w &= \wt V^{\rm L}_w / K^{\rm L}_w.
\end{align*}
Adopting the same terminology, $V^{\rm L}_w$ is called a left cell representation. Similarly, we also have $V^{\rm R}_w, V^{\rm LR}_w$, which are called a right cell representation and double cell representation, respectively.

Thus
\begin{equation} \label{e:cx-left}
\msc R (\lambda, \xi) \simeq \bigoplus _{ \text{left cells} } V^{\rm L}_w
\end{equation}
as a left representation of $W$, and 
\begin{equation}
\msc R (\lambda, \xi) \simeq \bigoplus _{ \text{double cells} } V^{\rm LR}_w
\end{equation}
as a representation of $W\times W$. 

The following theorem illustrates an example of Theorem \ref{p:diagram} in the complex case when the infinitesimal characters are integral.

\begin{thm}[{\cite[Theorem 3.20]{BV6}}]
Fix $w\in W$. Then  the left cell representation $V^{\rm L}_w$  (see (\ref{e:cx-left})) contains a unique special representation $\sigma (w)\in \Irr(W)$ with multiplicity one. 
 Let $\mca{O}(w) = \mca{O}_{\rm Spr}(\sigma(w))$ be the nilpotent in $\mfr{g}^*$  associated with $\sigma$. Then for any dominant integral regular weights $\lambda$ and $-\xi$, one has
$$
{\rm AV}(\wt X(\lambda , w\xi)) =\wt{\mca O(w)}. 
$$ 
\end{thm}
In the theorem above, ${\rm AV}(\wt X(\lambda , w\xi))$ can be identified with the associated variety of  ${\rm LAnn}(\wt X(\lambda , w\xi) )$.

\begin{eg}
We recall some results from \cite{BTs2}. Let $G=\Spin_{2r}(\C)$, which is viewed as a real group. According  to \cite{Bar1}, a nilpotent orbit $\mca O$ can  be associated  with an infinitesimal character
$\lambda_{\mca O}$ which satisfies certain conditions (see \cite[\S 2.3]{Bar1}). The fact is that $\mca O$  is the minimal orbit 
 which can be the associated variety of a $(\mfr g,K)$-module with infinitesimal character $(\lambda_{\rm L}, \lambda _R)$, with $\lambda_{\rm L}$ and $\lambda_{\rm R}$ both conjugate to $\lambda_\mca O$. 
 
 We denote by $\mca U _G(\mca O,\lambda)$ the set of irreducible admissible representations of $G$ attached to $\mca O$ and $\lambda= \lambda _{\mca O}$. We consider the following two cases: 
% \begin{align*}
%\text{(a)} \  \mca O &=(32^{2m-2}1), \ r =2m  & \text{(b)} \ \mca O &=(32^{2k}1^{2r-4k-3}),  \  r=2m \text{ or } 2m+1 \\
% \lambda &= \Big(m-\frac{1}{2}, \dots, \frac{3}{2}, \frac{1}{2}\mid m-1,\dots, 1, 0\Big)  &\lambda&=\Big(k+\frac{1}{2}, \dots, \frac{3}{2}, \frac{1}{2}\mid r-k-2,\dots, 1, 0\Big)  \\
% W_{\lambda} &= W(D_m\times D_m) &  W_{\lambda }&= W(D_{k+1}\times D_{r-k-1})
%  \end{align*}
  \begin{align*}
&\qquad \text{(a)} & & \qquad \text{(b)} \\ 
 \mca O &=(32^{2m-2}1), \ r =2m  &  \ \mca O &=(32^{2k}1^{2r-4k-3}),  \  r=2m \text{ or } 2m+1 \\
 \lambda &= \Big(m-\frac{1}{2}, \dots, \frac{3}{2}, \frac{1}{2}\mid m-1,\dots, 1, 0\Big)  &\lambda&=\Big(k+\frac{1}{2}, \dots, \frac{3}{2}, \frac{1}{2}\mid r-k-2,\dots, 1, 0\Big)  \\
 W_{\lambda} &= W(D_m\times D_m) &  W_{\lambda }&= W(D_{k+1}\times D_{r-k-1})
  \end{align*}

In both cases, every $\pi\in \mca U _G(\mca O,\lambda)$ is of the form $\pi=\wt X (\lambda, -w\lambda)$ for some $w$. We verify that these representations are indeed
attached to $\mca O$ by Theorem \ref{p:diagram}.  Let $\pi\in \mca U _G(\mca O,\lambda)$ with $I_\pi ={\rm LAnn}(\pi)$.  Note that $I_\pi = \check{I} (-w\lambda)$, where the caret denotes the principal antiautomorphism of $U(\mfr g)$ 
(see \cite{BV6}).
It can be checked readily  that $w\alpha\in \Phi_{\lambda} ^+$  for all $\alpha\in \Delta_\lambda$, and hence $I_\pi$ has maximal $\tau$-invariant. Accordingly,
$\sigma_{I_\pi}=\varepsilon_{W_{\lambda}}$ by Proposition \ref{p:tau}.  

We compute 
\begin{equation} \label{e:j1}
j_{W_\lambda} ^W (\varepsilon _{W_\lambda}) = 
\begin{cases}
(\emptyset ; 2^m) & \text{ in case (a)}, \\
(\emptyset ; 2^{k+1} 1^{n-2k-2}) & \text{ in case (b)}.
\end{cases}
\end{equation}
In (\ref{e:j1}), the unordered pairs of partitions are representations of $W(D_r)$.  Furthermore, the nilpotent orbit associated to (\ref{e:j1}) in each case is 
\begin{equation}
\mca{O}_{\rm Spr}(j_{W_\lambda} ^W (\varepsilon _{W_\lambda}))=
\begin{cases}
(32^{2m-2}1) &  \text{ in case (a)},\\
(32^{2k}1^{2r-4k-3}) & \text{ in case (b)},
\end{cases}
\end{equation}
as desired. This verifies that 
$$
{\rm AV}(I_\pi) ={\rm AV}(\pi) = \wt{\mca{O}_{\rm Spr}(j_{W_\lambda}^W(\varepsilon_{W_{\lambda_{\mca O}}}))},
$$
where $I_\pi ={\rm LAnn}(\pi)$. 

There is an analogue of real groups regarding these nilpotent orbits. See \cite{BTs1} for more details.
\end{eg}

\subsection{Real case} 
Let  $\wt{G}$ be the nonlinear double cover of the real points $G$ of a simply connected, semisimple complex Lie group. Here $G$ may not be a split real group. In \cite{Tsa}, a set of irreducible small representations of $\wt{G}$ with a certain infinitesimal character $\lambda$ is introduced, which we denote by
 $$\textstyle \prod _\lambda ^s (\wt{G}).$$
If $\wt{G}$ is simply laced, then $\lambda =\rho/2$; otherwise, see \cite[Table 1]{Tsa} for the tabulation of $\lambda$.  The condition for being small 
 is characterized by the ``maximal $\tau$-invariant" property, which implies that the coherent continuation representation of $W_\lambda$ acts on $\pi\in \prod_\lambda ^s (\wt{G})$ by the sign character $\varepsilon_{W_\lambda}$.
By Vogan's theory (see \cite{IC3}),  the $\tau$-invariant defined on  an irreducible  representation $\pi$ coincides with the $\tau$-invariant defined on its associated primitive ideal Ann$(\pi)$ (see \S \ref{s:prim}). 
 It follows from Proposition \ref{p:tau} that for $\pi\in \prod_\lambda^s (\wt{G})$, the primitive ideal $I_\pi :={\rm Ann}(\pi)$ has the Goldie rank representation $\sigma_{I_\pi} =\varepsilon_{W_\lambda}$. 
 Consequently, we have 
 $${\rm AV}(I_\pi) = \wt{\mca O}   \text{ with } \mca O= \mca{O}_{\rm Spr} (j_{W_\lambda}^W (\varepsilon_{W_\lambda}) ), $$ 
 by Theorem \ref{p:diagram}. See \cite[Table 1]{Tsa} for the list of such $\mca O$'s.  The above discussion thus assigns to every $\pi \in \prod _\lambda ^s (\wt{G})$ a complex nilpotent orbit $\mca O$.

\subsubsection{The split case} \label{SSS:R-spl}
We further assume that $G$ is the split real form of a connected, semisimple, simply connected complex group $G_\C$. Such $G$ admits a unique nonlinear double cover $\wt{G}$. In \cite{ABPTV}, the  pseudospherical principal series representation is defined as (following the notations in loc. cit.)
 $$
 I(\tilde{\delta}, \nu) = \text{Ind}_{\wt B} ^{\wt{G}} (\tilde{\delta}\otimes e^\nu),
 $$
where $\wt{B}= \wt{M} A^0 N$ is the covering of the Borel subgroup $B=MA^0N$ of $G$, $\tilde{\delta}$ is a genuine representation of $\wt{G}$, and $e^\nu$ is a character of $A^0$. 
Note that $M=B\cap K \cong \Z_2^n$ with $n$ the rank of $G$, whereas $\wt{M}$ is not abeliean in general. Write $J(\tilde{\delta}, \nu)$ for the unique irreducible quotient of $I(\tu\delta,\nu)$. Here $J(\tilde{\delta}, \nu)$ is just the theta representation $\Theta(\pi^\dag, \nu)$ discussed earlier in the paper.

Since $G$ is split, every $J(\tilde{\delta}, \lambda)$ discussed in \cite{ABPTV} is contained in  $\prod _\lambda^s (\wt{G})$. Furthermore, with the additional assumptions that $G$ is simply-laced, all representations in $\prod _\lambda^s (\wt{G})$
 are constructed from $J(\tilde{\delta},\rho/2)$ by applying the Cayley transforms, see \cite{Tsa}. In any case, we have a nilpotent orbit $\mca O$ naturally associated to $J(\tilde{\delta}, \lambda)$.  As mentioned, this actually motivated the formulation of Conjecture \ref{C:main}.  We remark that there are small representations of $\wt{G}$ attached to $\mca O$ other than those $J(\tilde{\delta}, \lambda)$ studied in \cite{ABPTV}.

 %In fact, by \cite{Vog4}, a representation has maximal $\tau$-invariant has Gelfand-Kirillov dimension equal to $|\Delta ^+|-|\Delta ^+(\lambda)|$, and hence the dimension of $AV(I_\pi)$ is 
 %$2 (|\Delta ^+|-|\Delta ^+(\lambda)|)$. For every type, there is a unique nilpotent orbit which has this dimension (see Table \cite{Tsa}). This verifies the diagram:
%\begin{center}
%\begin{tikzpicture}
%\tikzset{node distance=2cm, auto}
 % \node (P) {$\mathcal{O}$};
%  \node (B) [right of=P] {$j _{W_\lambda} ^W (\epsilon)$};
 % \node (A) [below of=P] {$I$};
 % \node (C) [below of=B] {$\epsilon _{W_\lambda}$};
 % \draw[->] (A) to node {AV} (P);
%  \draw[->] (P) to node [swap] {spr} (B);
 % \draw[->] (A) to node [swap] {$\sigma$} (C);
%  \draw[->] (C) to node {$j$} (B);
%\end{tikzpicture} 
%\end{center} 

\subsubsection{Non-split case}
The examples that we have considered so far are special cases of Theorem \ref{p:diagram}. One important feature is that the representation of $W_\lambda$ arising is $\varepsilon_{W_\lambda}$, which follows from the composite:
$$
\pi \mapsto I_\pi \mapsto \sigma_I=\varepsilon_{W_\lambda}. 
$$
Now we give another example of Theorem \ref{p:diagram}, for which the representation $\sigma_I$ is not the sign character. In a certain sense, this example lies beyond the scope of Conjecture \ref{C:main} and its real analogue. This example concerns covers of not necessarily split groups.

\begin{eg}
We recall some main results from \cite{Tra1}. Consider 
$$G=\wt{\Spin}(2m, 2l-2m),$$ 
 the (nonlinear) universal cover of the identity component of $\SO(2m, 2l-2m)$, with $2\lest m\lest l/2$. 
In \cite{Tra1}, for $s\gest 0$, a series of irreducible representations $\pi_s '$ are constructed as derived functor modules. The infinitesimal character of $\pi ' _s$ is 
$$
\nu_s = \Big(\underbrace{0,1,\dots,l-m-1}_{l-m},  \underbrace{ \val{\frac{s}{2}-m}, \val{\frac{s}{2} -m+1},\dots,  \val{\frac{s}{2}-1} }_{m} \Big). 
$$
Here $\nu _s$ is integral if $s$ is even; otherwise, the integral Weyl group of $\nu_s$ is of type $D_m\times D_{l-m}$.  Define the following complex nilpotent orbits: 
$$ \mca{O}(s)=
\begin{cases}
 (3^m1^{2l-3m}) & \text{ if }  s\gest m,\\
 (3^s 2^{2m-2s}1^{2l-4m+s}) & \text{ if }  0\lest s\lest m \text{ and } (l,s) \neq (2m,0),\\
 (2^{2m-2}1^4 ) & \text{ if }  s=0 \text{ and } l=2m. \\
\end{cases}
$$
Then ${\rm AV}(I_{\pi_s'}) =\wt{\mca O(s)}$, where $I_{\pi_s'}:=\text{Ann}(\pi_s ')$.  This fits into the diagram in Theorem \ref{p:diagram} as follows.
\begin{itemize}
\item If $s$ is even, then $\nu_s =  \frac{1}{2}h(\mca O^{\vee})$, which is integral.  Here $\mca O^{\vee}$ is the nilpotent orbit in the dual algebra $\mfr g^\vee$ such that $d(\mca O^\vee) = \mca O (s)$, where
$d: \mca N^\vee \to \mca N$ is the duality map of Spaltenstein (see the appendix in \cite{BV6}, for example), and $h(\mca O^\vee)$ is the semisimple element of a Jacobsen--Morozov triple for $\mca O^\vee.$
We have the commutative triangle
$$\begin{tikzcd}
 & \mca{O}(s)  \ar[rd, "{{\rm Spr}_\mbm{1}^{-1}}"]\\
 I_{\pi_s'} \ar[ru, "{{\rm AV}}"] \ar[rr, "\sigma"] & & \sigma_I.
\end{tikzcd}$$
If $0\lest s \lest m+1$, then the primitive ideal $I_{\pi ' _s}$ is the maximal ideal at infinitesimal character $\nu_s$. In such case, $I_{\pi_s '}$ is called special unipotent, and so is $\pi_s '$. 
\item If $s$ is odd, then $\nu_s$ is nonintegral and $W_{\nu_s} = W(D_m\times D_{l-m})$.   We have 
$$\begin{tikzcd}
  \mca{O}(s) \ar[r, "{{\rm Spr}_\mbm{1}^{-1}}"]    &  {\rm Spr}_\mbm{1}^{-1}(\mca O(s ) )  \\
  I_{\pi'_s} \arrow[r, "\sigma"]  \arrow[u , "{{\rm AV}}"]&   {\rm Spr}_\mbm{1}^{-1}(2^{s-1}1^{2m-2s+2} )\boxtimes {\rm Spr}_\mbm{1}^{-1}(1^{2l-2m}) \arrow[u, "{j_{W_{\nu_s}}^W}"'] .
  \end{tikzcd}$$
\end{itemize}
Note that it is clear in the second case above, the representation ${\rm Spr}_\mbm{1}^{-1}(2^{s-1}1^{2m-2s+2} )$ is not the sign character of $W(D_m)$ in general.
\end{eg}

%%%
\section{Some remarks} \label{S:rmk}
The main Conjecture \ref{C:main} is stated for $\Theta(\pi^\dag, \nu)$ for $p$-adic $F$ only in the tame case. One reason is that for $p\nmid n$, the set $\msc{X}_{Q,n}$ is the ``moduli space" of the space of Whittaker functional of the genuine principal series $I(\pi^\dag, \nu)$; in particular,
$$\dim \Wh_\psi(I(\pi^\dag, \nu)) = \val{\msc{X}_{Q,n}}.$$
However, if $p|n$, then \eqref{E:main1} is still expected to hold, but \eqref{E:main2} might fail.

It is possible to incorporate the archimedean counterpart (as discussed in \S \ref{SSS:R-spl}) into the statement of Conjecture \ref{C:main}, at least regarding \eqref{E:main1}, if we assume a Barbasch--Vogan character expansion for covering groups. However, since the discussion in \S \ref{S:arch} utilizes the algebraic invariants instead, the results in \S \ref{SSS:R-spl} prove only a natural archimedean analogue of Conjecture \ref{C:main}. It is desirable to fill this ``gap" by checking the following:
\begin{enumerate}
\item[--] the Barbasch--Vogan character expansion holds for covering groups, as mentioned in \S \ref{s:WFAV}, and thus one has the analytic invariants $\mca{N}_{\rm tr}(\pi), \mca{N}_{\rm tr}^{\rm max}(\pi)$ for genuine representation $\pi$; and
\item[--] as in the linear algebraic case, these analytic invariants agree with the their algebraic counterpart.
\end{enumerate}

It is also possible to unify the $p$-adic and archimedean cases in the statement of Conjecture \ref{C:main} by using the generalized or degenerate Whittaker module of $\Theta(\pi^\dag, \nu)$, which we recall is defined for all local fields. Indeed, to every nilpotent orbit $\mca{O} \subset \mfr{g}$ and genuine irreducible representation $\pi$ of $\wt{G}$, one can associate a certain generalized Whittaker space $\pi_\mca{O}$ (see \cite{JLS, GGS17, GGS, GoSa}) and thus define
$$\mca{N}_{\rm Wh}(\pi)=\set{\mca{O} \in \mca{N}: \pi_\mca{O} \ne 0}.$$
Consider the subset $\mca{N}_{\rm Wh}^{\rm max}(\pi) \subset \mca{N}_{\rm Wh}(\pi)$ consisting of maximal elements. Then Conjecture \ref{C:main} could be stated by using $\mca{N}_{\rm Wh}^{\rm max}$ instead.
The advantage of this perspective is that it applies to all local field $F$ and covers $\wt{G}$ simultaneously. In the $p$-adic setting, its relation with $\mca{N}_{\rm tr}^{\rm max}(\pi)$ is the content of \cite{MW1, Var1, Li3, Pate}. However, to the best of our knowledge, in the archimedean setting, the relation between $\mca{N}_{\rm Wh}(\pi)$ and $\mca{N}_{\rm alg}(\pi)$ has not been understood completely, even for linear group $G$. We refer the reader to the excellent exposition \cite{GoSa} and the references therein for a  review on recent advance and open questions on the theory of generalized Whittaker space.

%%%%%%%%%%%%%%% %%%%%%%%%%%%%%%%%%%%%%%%%%%%%%%%%%%%%%%%%

\end{document}